\documentclass[a4paper,11pt]{article}
\usepackage{stmaryrd}
\usepackage{amsfonts}
\usepackage{bbm}

\usepackage{amscd}
\usepackage{mathrsfs}
\usepackage{latexsym,amssymb,amsmath,amscd,amscd,amsthm,amsxtra,xypic}
\usepackage[dvips]{graphicx}
\usepackage[utf8]{inputenc}
\usepackage[T1]{fontenc}
\usepackage{enumerate}
\usepackage{lmodern}
\usepackage{amssymb}
\usepackage[all]{xy}
\usepackage{ifpdf}
\ifpdf
\usepackage[colorlinks=true,linkcolor=blue,final,backref=page,hyperindex,citecolor=red]{hyperref}
\else
\usepackage[colorlinks,final,backref=page,hyperindex,hypertex]{hyperref}
\fi

\usepackage{nicefrac,mathtools,enumitem}
\usepackage{microtype}
\usepackage{amsfonts}
\usepackage{amssymb}
\usepackage{mathrsfs}
\usepackage{amsmath}
\usepackage{amsthm}
\usepackage{enumerate}
\usepackage{indentfirst}
\usepackage{amsfonts}
\usepackage{amssymb}
\usepackage{mathrsfs}
\usepackage{amsmath}
\usepackage{amsthm}
\usepackage{enumerate}
\usepackage{cite}
\usepackage{mathrsfs}
\usepackage{kpfonts}
\usepackage{geometry}
\allowdisplaybreaks

\newcommand{\K}{{\mathbb{K}}}
\newcommand{\dl}{{\displaystyle}}

\newtheorem{df}{Definition}[section]
\newtheorem{thm}{Theorem}[section]
\newtheorem{cor}{Corollary}[section]
\newtheorem{rem}{Remark}[section]
\newtheorem{rems}{Remarks}[section]
\newtheorem{prop}{Proposition}[section]
\newtheorem{exa}{Example}[section]

\newtheorem{lem}{Lemma}[section]

\begin{document}
\date{}
\title{\bf Cohomology and formal deformations of $n$-Hom-Lie color algebras}\author{\bf K. Abdaoui, R. Gharbi,  S. Mabrouk, A. Makhlouf}
\author{{ K. Abdaoui$^{1}$
 \footnote { E-mail: Abdaouielkadri@hotmail.com}
,\  R. Gharbi$^{1,3}$
    \footnote { E-mail:  Rahma.Gharbi@uha.fr}
,\  S. Mabrouk$^{2}$
 \footnote { E-mail: mabrouksami00@yahoo.fr }
\ and A. Makhlouf$^{3}$
 \footnote { E-mail: abdenacer.makhlouf@uha.fr $($Corresponding author$)$}
}\\
{\small 1.  University of Sfax, Faculty of Sciences Sfax,  BP
1171, 3038 Sfax, Tunisia} \\
{\small 2.  University of Gafsa, Faculty of Sciences Gafsa, 2112 Gafsa, Tunisia}\\
{\small 3.~ Universit\'e de Haute Alsace, IRIMAS - Département de Mathématiques,}\\ {\small  18, rue des frères Lumière,
F-68093 Mulhouse, France}}
\date{}

 \maketitle{}
 %\keywords{Constructions and Cohomology of Hom pre-Lie superalgebras}%
 % ----------------------------------------------------------------
\numberwithin{equation}{section}
 \begin{abstract}
The aim of this paper is to provide a cohomology of $n$-Hom-Lie color algebras governing one parameter formal  deformations. Then,  we study formal deformations of a $n$-Hom-Lie color algebra and introduce the notion of  Nijenhuis operator on an $n$-Hom-Lie color algebra, which could give rise to infinitesimally  trivial $(n-1)$-order deformations. Furthermore, in connection with Nijenhuis operators we introduce and discuss  the notion of a product  structure on  $n$-Hom-Lie color algebras.
\end{abstract}
\noindent\textbf{Keywords}: $n$-Hom-Lie color algebra,  cohomology, deformation,  Nijenhuis operator, almost structure,  product  structure.  \\
\noindent\textbf{2020 MSC}:  17D30, 17B61, 17A40, 17B56, 53C15.
\maketitle
% ----------------------------------------------------------------
\tableofcontents
\section{Intoduction}~~

The generalization of Lie algebra, which is now known as
Lie color algebra is introduced by Ree \cite{Ree}. This class  includes Lie superalgebras which are $\mathbb{Z}$-graded and play an important role in supersymmetries. More generally, Lie color algebra plays an important role in theoretical physics, as
explained in \cite{SuZhao1, SuZhao2}. Montgomery \cite{Montgomery} proved that Simple Lie color algebra can be obtained
from associative graded algebra, while the Ado theorem and the PBW theorem of Lie
color algebra were proven by Scheunert \cite{scheunert1979generalized}. In the last two decades, Lie color algebras have
developed as an interesting topic in Mathematics and Physics $($see \cite{ArmakanSilvestrovFarhangdoost, BeitesKaygorodovPopov, BergenPassman, Feldvoss, KaygorodovPopov, Wilson} for more details$)$.

Ternary Lie algebras and more generally $n$-ary Lie algebras are natural generalization of binary Lie algebras, where one considers $n$-ary operations and a generalization of Jacobi condition. The most common generalization consists of expressing the adjoint map as a derivation. The corresponding algebras were called $n$-Lie algebras  and were first introduced and studied by Filippov in \cite{ChapBaSilnhomliecolor:Filippov:nLie}  and then completed  by Kasymov in  \cite{ChapBaSilnhomliecolor:Kasymov:nLie}. These algebras, in the ternary cas, appeared in  the mathematical algebraic foundations of Nambu mechanics developed by Takhtajan and Daletskii in  \cite{ChapBaSilnhomliecolor:DalTakh,ChapBaSilnhomliecolor:Takhtajan:foundgenNambuMech,ChapBaSilnhomliecolor:Takhtajan:cohomology}, as a genralization of Hamiltonian Mechanics involving more than one hamiltonian.   Besides Nambu mechanics, $n$-Lie algebras revealed to have many applications in physics. The second approach consists of generalizing Jacobi condition as a summation  over $S_{2n-1}$ instead of $S_3$.

Hom-type generalizations of $n$-ary algebras were considered first in  \cite{ChapBaSilnhomliecolor:AtMaSi:GenNambuAlg}, where  $n$-Hom-Lie algebras and other $n$-ary Hom-algebras of Lie type and associative type were introduced. The usual identities  are twisted by linear maps. As a particular case one recovers Hom-Lie algebras which were motivated by quantum deformations of algebras of vector fields like Witt and Virasoro algebras.
Further properties, construction methods, examples, cohomology and central extensions of $n$-ary Hom-algebras have been considered in \cite{ChapBaSilnhomliecolor:ams:ternary, ChapBaSilnhomliecolor:ams:n, ChapBaSilnhomliecolor:akms:ternary}.
%The construction of $(n+1)$-Lie algebras induced by $n$-Lie algebras using combination of bracket multiplication with a trace, motivated by the work of Awata  \cite{ChapBaSilnhomliecolor:almy:quantnambu} on the quantization of the Nambu brackets, was generalized using the brackets of general Hom-Lie algebra or $n$-Hom-Lie and trace-like linear forms depending on the linear maps defining the Hom-Lie or $n$-Hom-Lie algebras \cite{ChapBaSilnhomliecolor:ams:ternary,ChapBaSilnhomliecolor:ams:n}.

The (co)homology theory
with adjoint representation for $n$-Lie algebras was introduced by Takhtajan in \cite{ChapBaSilnhomliecolor:Takhtajan:cohomology,ChapBaSilnhomliecolor:DalTakh} and
by Gautheron in \cite{Gautheron-P} from deformation theory viewpoint. The general cohomology theory for $n$-Lie
algebras, Leibniz $n$-algebras was established in \cite{ChapBaSilnhomliecolor:RM, ChapBaSilnhomliecolor:JM}  and $n$-Hom-Lie algebras and superalgebras in \cite{Cohomology-super, AmmarMabroukMakhlouf, AmmarNajib}.

Inspired by these works, we aim to  study the cohomology and deformation of graded $n$-Hom-Lie algebras. Moreover, we consider  the notion of Nijenhuis operator in connection with the  study  of $(n-1)$-order deformation of graded $n$-Hom-Lie algebras. In particular, we discuss the notion of product structure.

This paper is organized as follows. In Section \ref{basics}, we recall some basic definitions on $n$-Hom-Lie color algebras. Section \ref{constructions} is devoted to various constructions of $n$-Hom-Lie  color algebras  and  Hom-Leibniz color algebras. Furthermore, we introduce the notion of a representation of a  $n$-Hom-Lie  color algebra and construct the corresponding semi-direct product.  In Section \ref{cohomology}, we  study   cohomologies  with respect to   given representations. In Section \ref{deformations},  we discuss formal and  infinitesimal deformations of a  $n$-Hom-Lie color algebra. Finally, in Section \ref{nijenhuis}, we introduce the notion of Nijenhuis operators , which is connected  to infinitesimally trivial $(n - 1)$-order
deformations. Moreover,   we define  the   product structure on $n$-Hom-Lie color algebra  using  Nijenhuis conditions.

%Throughout this paper $\mathbb{K}$ is a field of characteristic $0$ and all vector spaces are over $\mathbb{K}$.
 \section{Basics on   $n$-Hom-Lie  color algebras}\label{basics}
 This Section contains  preliminaries and definitions on graded spaces, algebras and $n$-Hom-Lie color algebras which is the graded case of $n$-Hom-Lie algebras $($see  \cite{AmmarImenSamiNacer, Multiplicative-S, Zhang} for more details$)$.

Throughout this paper  $\K$ will denote a commutative field of characteristic zero, $\Gamma$ will stand for an abelian group. A vector space $\mathfrak{g}$ is said to be a $\Gamma$-graded if we are given a family $(\mathfrak{g}_{\gamma})_{\gamma\in \Gamma}$ of sub-vector space of $\mathfrak{g}$ such that $\mathfrak{g}=\bigoplus_{\gamma\in \Gamma}\mathfrak{g}_{\gamma}.$
An element $x\in \mathfrak{g}_{\gamma}$ is said to be homogeneous of degree $\gamma$. If the base field is considered as a graded vector space, it is understood that the graduation of $\K$ is given by
${\K}_0 = \K  \ \ \ \mbox{and} \ \ \ {\K}_{\gamma}= \left\{0\right\}, \ \ \mbox{if}\ \ \gamma\in \Gamma\setminus \left\{ 0\right \}.$
Now, let $\mathfrak{g}$ and $\mathfrak{h}$ be two $\Gamma$-graded vector spaces. A linear map $f: \mathfrak{g} \longrightarrow \mathfrak{h}$ is said to be homogeneous of degree $\xi\in \Gamma$, if $f(x)$ is homogeneous of degree $\gamma + \xi$ whenever the element $x\in \mathfrak{g}_{\gamma}$. The set of all  linear maps of degree $\xi$ will be denoted   by ${Hom(\mathfrak{g},\mathfrak{h})}_{\xi}$. Then, the vector space of all linear maps of $\mathfrak{g}$ into $\mathfrak{h}$  is $\Gamma$-graded and denoted by $Hom(\mathfrak{g},\mathfrak{h})=\bigoplus_{\xi\in \Gamma}{Hom(\mathfrak{g},\mathfrak{h})}_{\xi}$.

We mean by algebra (resp. $\Gamma$-graded algebra ) $(\mathfrak{g},\cdot)$ a vector space (resp. $\Gamma$-graded vector space) with multiplication, which  we denote by the concatenation, such that $\mathfrak{g}_{\gamma}\mathfrak{g}_{\gamma'}\subseteq \mathfrak{g}_{\gamma+\gamma'}$, for all $\gamma, \gamma'\in \Gamma$. In the graded case, a map $f: \mathfrak{g} \longrightarrow \mathfrak{h}$, where $\mathfrak{g}$ and $\mathfrak{h}$ are $\Gamma$-graded algebras, is called a $\Gamma$-graded algebra homomorphism  if it is a degree zero algebra homomorphism.

We mean by Hom-algebra (resp. $\Gamma$-graded Hom-algebra) a triple  $(\mathfrak{g},\cdot, \alpha)$ consisting of a vector space (resp. $\Gamma$-graded vector space), a multiplication and an endomorphism  $\alpha$ (twist map)  (resp. of  degree zero endomorphism  $\alpha$).

 For more detail about graded algebraic structures, we refer to \cite{scheunert1979generalized}. In the following, we recall the definition of bicharacter on an abelian group $\Gamma$.
\begin{df}
Let $\Gamma$ be an abelian group. A map $\varepsilon:\Gamma\times\Gamma\rightarrow\K\setminus \left\{0\right\}$ is called a \textbf{bicharacter} on ${\Gamma}$ if the following identities are satisfied
\begin{align}
\label{condi1-bicharacter}&\varepsilon(\gamma_1,\gamma_2)\varepsilon(\gamma_2,\gamma_1)=1,\\
\label{condi2-bicharacter}&\varepsilon(\gamma_1,\gamma_2+\gamma_3)=\varepsilon(\gamma_1,\gamma_2)\varepsilon(\gamma_1,\gamma_3),\\
\label{condi3-bicharacter}&\varepsilon(\gamma_1+\gamma_2,\gamma_3)=\varepsilon(\gamma_1,\gamma_3)\varepsilon(\gamma_2,\gamma_3),\ \ \forall \gamma_1,\gamma_2,\gamma_3\in {\Gamma}.
\end{align}
\end{df}
In particular, the definition above implies the following relations
\begin{align*}
\varepsilon(\gamma,0)=\varepsilon(0,\gamma)=1,\ \varepsilon(\gamma,\gamma)=\pm1, \  \textrm{for\ all}\  \gamma\in\Gamma.
\end{align*}
Let $\mathfrak{g}=\dl\bigoplus_{\gamma \in \Gamma}\mathfrak{g}_\gamma$  be a  $\Gamma$-graded vector space.
If $x$ and $x'$ are two homogeneous elements in $\mathfrak{g}$ of degree $\gamma$ and $\gamma'$ respectively and $\varepsilon$ is a bicharacter, then we shorten the notation by writing $\varepsilon(x,x')$ instead of $\varepsilon(\gamma,\gamma')$. If $X=(x_1,\ldots,x_p)\in \otimes^p\mathfrak{g}$, we set  $$\varepsilon(x,X_i)=\varepsilon(x,\displaystyle\sum_{k=1}^{i-1}x_{k}) \ \ \text{for}\ \ i>1, \text{and}\ \ \varepsilon(x,X_i)=1 \ \ \text{for}\  \ i=1 ,$$
$$ \varepsilon(x,X^i)=\varepsilon(x,\displaystyle\sum_{k=i+1}^px_i)\ \ \text{for}\ \ i<p, \text{and}\ \ \varepsilon(x,X^i)=1 \ \ \text{for}\  \ i=p ,$$
$$ \varepsilon(x,X_i^j)=\varepsilon(x,\displaystyle\sum_{k=i}^{j}x_k).$$

  Then, we define the general linear Lie color algebra $gl(\mathfrak{g})=\dl\bigoplus_{\gamma\in\Gamma}gl(\mathfrak{g})_\gamma$, where
  $$gl(\mathfrak{g})_\gamma=\{f:\mathfrak{g}\rightarrow\mathfrak{g}/f(\mathfrak{g}_{\gamma'})\subset\mathfrak{g}_{\gamma+\gamma'}\ \textrm{and}\ \alpha \circ f =f \circ \alpha\  \textrm{for\ all}\  \gamma'\in\Gamma\}$$

In the following, we recall the notion of $n$-Hom-Lie color algebra given by I. Bakayoko and S. Silvestrov  in \cite{Multiplicative-S} which is a generalization of $n$-Hom-Lie superalgebra introduced in \cite{Cohomology-super}.
\begin{df}
A  \textbf{$n$-Hom-Lie color  algebra} is a graded vector space $\mathfrak{g}=\bigoplus_{\gamma\in \Gamma}\mathfrak{g}_{\gamma}$ with a multilinear map $ [\cdot ,..., \cdot] : \mathfrak{g}\times ... \times \mathfrak{g} \longrightarrow \mathfrak{g} $, a bicharacter  $\varepsilon:\Gamma\times\Gamma\rightarrow\K\setminus \left\{0\right\}$ and a linear map  $ \alpha : \mathfrak{g} \longrightarrow \mathfrak{g}$ of degree zero, such that
  \begin{align}\label{ColorSkewSym}
  & \big[x_1, \cdots,x_i,x_{i+1},\cdots,x_n\big]=  - \varepsilon(x_i,x_{i+1})\big[x_1, \cdots,x_{i+1},x_i,\cdots,x_n\big],
  \\&\nonumber\\\label{identity}
    &\big[ \alpha(x_1),\cdots, \alpha(x_{n-1}),[y_1,y_2,\cdots,y_n] \big ] = \nonumber   \\
   &\quad\quad\quad\quad\quad\quad\quad\quad\sum_{i=1}^{n}  \varepsilon(X,Y_i)\big[\alpha(y_1) ,\cdots, \alpha(y_{i-1}),[x_1,x_2,\cdots,y_i],\alpha(y_{i+1}),\cdots,\alpha(y_n)\big],
  \end{align}
%  where $x_i, y_j\in \mathfrak{g},$ $ X = \sum_{i=1}^{n-1}\bar x_i $, $ Y_i = \sum_{j=1}^{i} \bar y_{j-1} $, $y_0=e$.

 \end{df} The identity $\eqref{identity}$ is called \emph{$\varepsilon$-n-Hom-Jacobi identity} and the  Eq. $\eqref{ColorSkewSym} $ is equivalent to {\small \begin{align} \label{eq-antisymetric}
    &[x_1,\ldots,x_i,\ldots,x_j,\ldots,x_n] = - \varepsilon (x_i,X_{i+1}^{j-1}) \varepsilon(X_{i+1}^{j-1}, x_j) \varepsilon (x_i,x_j) [x_1,\ldots,x_j,\ldots,x_i,\ldots,x_n]. \end{align} }

Let $(\mathfrak{g}, [\cdot,\cdots,\cdot],\varepsilon, \alpha)$ and $(\mathfrak{g}^{'}, [\cdot,\cdots,\cdot]^{'},\varepsilon, \alpha^{'})$ be two $n$-Hom-Lie color algebras. A  linear map of degree zero $f: \mathfrak{g} \rightarrow \mathfrak{g}^{'}$ is a  $n$-Hom-Lie color algebra \textbf{morphism} if its satisfies
\begin{equation*}
    f \circ \alpha = \alpha^{'} \circ f,
\end{equation*}
\begin{equation*}
    f([x_1,\ldots,x_n]) =[f(x_1),\ldots,f(x_n)]^{'}.
\end{equation*}
\begin{df}
  Let $(\mathfrak{g}, [\cdot,\cdots,\cdot],\varepsilon, \alpha)$ be a  $n$-Hom-Lie color algebra. Its called
  \begin{itemize}
    \item \textbf{Multiplicative} $n$-Hom-Lie color algebra if $\alpha [x_1,\cdots,x_n] = [\alpha(x_1),\cdots,\alpha(x_n)]$.
    \item \textbf{Regular } $n$-Hom-Lie color algebra if $\alpha$ is an automorphism.
    \item \textbf{Involutive} $n$-Hom-Lie color algebra if $\alpha^2 = Id$.
  \end{itemize}
\end{df}
\begin{rems}\
  \begin{enumerate}
  \item When $\Gamma =\{0  \}$, the trivial group, this is the ordinary $n$-Hom-Lie algebra, see \cite{ChapBaSilnhomliecolor:AtMaSi:GenNambuAlg} for more details.
   \item When $\Gamma = \mathbb{Z}_{2}$, $ \varepsilon(x, y) = (-1)^{|x||y|}$, this is the $n$-Hom-Lie superalgebra defined in \cite{Cohomology-super}.
    \item If $n=2$ (resp, $n=3$) we recover Hom-Lie color algebra (resp. $3$-hom Lie color algebra).
    \item When $\alpha = Id$, we get $n$-Lie color algebra.
  \end{enumerate}
\end{rems}
\begin{df}Let $(\mathfrak{g}, [\cdot,\cdots,\cdot],\varepsilon, \alpha)$ be a  $n$-Hom-Lie color algebras. Then \begin{enumerate}
    \item

A $\Gamma$-graded subspace $\mathfrak{h}$ of  $\mathfrak{g}$ is a \textbf{color subalgebra} of $\mathfrak{g}$, if for all $\gamma_1,\gamma_2,\ldots,\gamma_n $ $\in \Gamma$ :
\begin{align*}
    \alpha(\mathfrak{h}_{\gamma_1}) \subseteq \mathfrak{h}_{\gamma_1}
\ \ \text{and}\ \
    [\mathfrak{h}_{\gamma_1},\mathfrak{h}_{\gamma_2},\ldots,\mathfrak{h}_{\gamma_n}] \subseteq \mathfrak{h}_{{\gamma_1}+\ldots+{\gamma_n}}.
\end{align*}
\item
A  \textbf{color ideal} $\mathfrak{I}$ of   $\mathfrak{g}$ is a color Hom-subalgebra of $\mathfrak{g}$ such that, for all $\gamma_1,\gamma_2,\ldots,\gamma_n $ $\in \Gamma$ :
\begin{align*}
    \alpha(\mathfrak{I}_{\gamma_1}) \subseteq \mathfrak{I}_{\gamma_1}
\ \ \text{and}\ \
    [\mathfrak{I}_{\gamma_1},\mathfrak{g}_{\gamma_2},\ldots,\mathfrak{g}_{\gamma_n}] \subseteq \mathfrak{I}_{{\gamma_1}+\ldots+{\gamma_n}}.
\end{align*}
\item The
\textbf{center} of $\mathfrak{g}$ is  the set  $Z(\mathfrak{g})=\{x \in \mathfrak{g} \ | \ [x,y_1,\ldots,y_{n-1}]=0, \forall y_1.\ldots y_n  \in \mathfrak{g}  \}$. It is easy to show  that $Z(\mathfrak{g})$ is a color ideal of $\mathfrak{g}$.\end{enumerate}
\end{df}
\section{Constructions and Representations of $n$-Hom-Lie color algebras}\label{constructions}
In this Section, we show   some constructions of $n$-Hom-Lie  color algebras  and  Hom-Leibniz color algebras associated to $n$-Hom-Lie  color algebras. Moreover,
 we introduce the notion of a representation of  $n$-Hom-Lie  color algebra and construct the
corresponding semi-direct product.
\subsection{Yau-Twist of $n$-Hom-Lie color algebras}
In the following theorem, starting from a $n$-Hom-Lie color algebra and a $n$-Hom-Lie color algebra endomorphism, we construct a new $n$-Hom-Lie color algebra. We say that it is obtained by  Yau twist.
\begin{thm}\label{twisted-thrm}
   Let $(\mathfrak{g}, [\cdot,\cdots,\cdot],\varepsilon, \alpha)$ be a $n$-Hom-Lie color algebra and $\beta : \mathfrak{g} \longrightarrow \mathfrak{g}$ be a n-Hom-Lie color algebra endomorphism of degre zero. Then $(\mathfrak{g}, [\cdot,\cdots,\cdot]_\beta,\varepsilon,\beta \circ \alpha)$, where $ [\cdot,\cdots,\cdot]_\beta = \beta \circ [\cdot,\cdots,\cdot]$, is a n-Hom-Lie color algebra. \\
   Moreover, suppose that $(\mathfrak{g}^{'}, [\cdot,\cdots,\cdot]^{'},\varepsilon ,\alpha')$ is a $n$-Hom-Lie color algebra and  $\beta^{'} : \mathfrak{g}^{'} \longrightarrow \mathfrak{g}^{'} $ is a $n$-Hom-Lie color algebra endomorphism. If $f :\mathfrak{g} \longrightarrow \mathfrak{g}^{'}$ is a $n$-Hom-Lie color algebra morphism that satisfies $f\circ \beta = \beta^{'} \circ  f$, then
   $$ f : (\mathfrak{g}, [\cdot,\cdots,\cdot]_\beta,\varepsilon,\beta \circ \alpha) \longrightarrow (\mathfrak{g}^{'}, [\cdot,\cdots,\cdot]^{'}_{\beta^{'}},\varepsilon, \beta' \circ \alpha') $$
   is a morphism of $n$-Hom-Lie color algebras.
\end{thm}
\begin{proof}
Obviously $[\cdot,\cdots,\cdot]_{\beta}$ is a $\varepsilon$-skew-symmetric. Furthermore,  $(\mathfrak{g}, [\cdot,\cdots,\cdot]_\beta,\varepsilon,\beta \circ \alpha)$ satisfies the $\varepsilon$-$n$-Hom-Jacobi identity $\eqref{identity}$, indeed
\begin{align*}
    & [\beta \circ \alpha (x_1),\cdots,\beta \circ \alpha (x_{n-1}),[y_1,\cdots,y_n]_\beta]_\beta \\
    = &  \beta^2 ([\alpha (x_1),\cdots,\alpha (x_{n-1}),[y_1,\cdots,y_n]]) \\
    =&   \sum_{i=1}^{n}  \varepsilon(X,Y_i) \beta^2 (\big[ \alpha(y_1) ,\cdots,  \alpha(y_{i-1}),[x_1,x_2,\cdots,y_i] \alpha(y_{i+1}),\cdots, \alpha(y_n)\big])\\=&   \sum_{i=1}^{n}  \varepsilon(X,Y_i) \big[\beta \circ \alpha(y_1) ,\cdots, \beta \circ \alpha(y_{i-1}),[x_1,x_2,\cdots,y_i]_\beta,\beta \circ \alpha(y_{i+1}),\cdots,\beta \circ \alpha(y_n)\big]_\beta.
\end{align*}
The second assertion follows from
\begin{align*}
    f\big{(} [x_1,\cdots,x_n]_\beta\big{)} = & [f \circ\beta(x_1),\cdots,f \circ\beta(x_n)]^{'} = [\beta^{'} \circ f (x_1), \cdots, \beta^{'} \circ f (x_n)]^{'} = [  f (x_1), \cdots, f (x_n)]^{'}_{\beta^{'}}.
\end{align*}
The proof is finished.
\end{proof}
In particular, we have the following construction of $n$-Hom-Lie color algebra using $n$-Lie color algebras and algebra morphisms.
\begin{cor}\label{twist}
 Let $(\mathfrak{g}, [\cdot,\cdots,\cdot],\varepsilon)$ be a $n$-Lie color algebra and $\alpha : \mathfrak{g} \longrightarrow \mathfrak{g}$ be a  n-Lie color algebra endomorphism. Then $(\mathfrak{g}, [\cdot,\cdots,\cdot]_\alpha ,\varepsilon,\alpha)$, where $ [\cdot,\cdots,\cdot]_\alpha = \alpha\circ [\cdot,\cdots,\cdot]$, is a n-Hom-Lie color algebra.
\end{cor}
\begin{exa}\label{Exple} \cite{Multiplicative-S}
  Let $\Gamma=\mathbb{Z}_2\times\mathbb{Z}_2,$ and $ \varepsilon((i_1, i_2), (j_1, j_2))=(-1)^{i_1j_2-i_2j_1}$. Let  $L$ be a $\Gamma$-graded vector  space $$\mathfrak{g}=\mathfrak{g}_{(0, 0)}\oplus \mathfrak{g}_{(0, 1)}\oplus \mathfrak{g}_{(1, 0)}\oplus \mathfrak{g}_{(1, 1)}$$ with
$\mathfrak{g}_{(0, 0)}=<e_1, e_2>,\quad  \mathfrak{g}_{(0, 1)}=<e_3>,\quad  \mathfrak{g}_{(1, 0)}=<e_4>,\quad \mathfrak{g}_{(1, 1)}=<e_5>.$\\
The  bracket $[\cdot, \cdot, \cdot, \cdot] : \mathfrak{g}\times \mathfrak{g}\times \mathfrak{g}\times \mathfrak{g}\rightarrow \mathfrak{g}$  defined with respect to  basis $\{e_i \mid i=1, \dots ,5\}$
by
\begin{gather*}
[e_2, e_3, e_4, e_5]=e_1,\ \  [e_1, e_3, e_4, e_5]=e_2, \ \  [e_1, e_2, e_4, e_5]=e_3, \\
[e_1, e_2, e_3, e_4]=0,\ \  [e_1, e_2, e_3, e_5]=0
\end{gather*}
makes $\mathfrak{g}$ into a five dimensional $4$-Lie color algebra.

Now, we define  a morphism $\alpha :\mathfrak{g}\rightarrow \mathfrak{g}$ by:
$$\alpha(e_1)=e_2,\quad \alpha(e_2)=e_1, \quad \alpha(e_i)=e_i, \quad i=3, 4, 5.$$
Then $\mathfrak{g}_\alpha=(\mathfrak{g}, [\cdot, \cdot, \cdot, \cdot]_\alpha, \varepsilon, \alpha)$ is a $4$-Hom-Lie color algebra, where the new brackets are given as
\begin{gather*}
[e_2, e_3, e_4, e_5]_\alpha=e_2,\ \  [e_1, e_3, e_4, e_5]_\alpha=e_1, \ \  [e_1, e_2, e_4, e_5]_\alpha=e_3, \\
[e_1, e_2, e_3, e_4]_\alpha=0,\ \  [e_1, e_2, e_3, e_5]_\alpha=0.
\end{gather*}
\end{exa}
  Let $ \mathcal{G}= (\mathfrak{g}, [\cdot,\cdots,\cdot],\varepsilon, \alpha)$ be a multiplicative $n$-Hom-Lie color algebra and $ p \geq 0$. Define the $p$th derived  $\mathcal{G}$ by:
  $$ \mathcal{G}^{p} = (\mathfrak{g},[\cdot,\cdots,\cdot]^{p}= \alpha^{2p-1} \circ [\cdot,\cdots,\cdot],\varepsilon, \alpha^{2p}) . $$
  Note that $\mathcal{G}^{0} = \mathcal{G}$, $\mathcal{G}^{1}= (\mathfrak{g},[\cdot,\cdots,\cdot]^{1}= \alpha \circ [\cdot,\cdots,\cdot] ,\varepsilon, \alpha^{2})$.

\begin{cor}
  With the above notations, the $p$th derived   $ \mathcal{G}$ is also a $n$-Hom-Lie color algebra for each $ p \geq 0$.
\end{cor}
\begin{lem} \label{lemme}
 Let $(\mathfrak{g},[\cdot,\cdots,\cdot],\varepsilon,\alpha)$ be  a  $n$-Hom-Lie color algebra and $\mathfrak{h}$ be a $\Gamma $-graded vector space. If there exists a bijective linear map of degree zero $f: \mathfrak{h} \to \mathfrak{g} $, then $(\mathfrak{h},[\cdot,\cdots,\cdot]',\varepsilon, f^{-1} \circ\alpha\circ f )$ is  a  $n$-Hom-Lie color algebra, where the  $n$-ary bracket  $[\cdot,\ldots,\cdot]'$ is defined  by $$[x_1,\ldots,x_n]' = f^{-1} \circ [f(x_1),\ldots,f(x_n)],\ \forall \ x_i \in \mathfrak{h}.$$
  Moreover, $f$ is an algebra isomorphism.
 \end{lem}
 \begin{proof}
 Straightforward.
 \end{proof}

\begin{prop}
   Let $(\mathfrak{g}, [\cdot,\cdots,\cdot],\varepsilon, \alpha)$ be a regular $n$-Hom-Lie color algebra, then $(V, [\cdot,\cdots,\cdot]_{\alpha^{-1}}= \alpha^{-1} \circ [\cdot,\cdots,\cdot],\varepsilon)$ is a $n$-Lie color algebra.
\end{prop}
\begin{cor}
Let $(\mathfrak{g}, [\cdot,\cdots,\cdot],\varepsilon, \alpha)$ be a $n$-Hom-Lie color algebra with $n \geq 3$. Let  $a_1,\ldots,a_p\in \mathfrak{g}_0$ such that $\alpha(a_i)=a_i$,$\forall i \in \{1,\ldots,p\}$. Then,  $(\mathfrak{g}, \{\cdot,\cdots,\cdot\},\varepsilon, \alpha)$ is an $(n-p)$-Hom-Lie color algebra, where
\begin{equation}
    \{x_1,...,x_{n-p}\}=[a_1,\ldots,a_p,x_1,...,x_{n-p}],\ \ \forall x_1,...,x_{n-p}\in \mathfrak{g}.
\end{equation}
\end{cor}

\subsection{From $n$-Hom-Lie color algebras to Hom-Leibniz color algebras}

In the following, we recall the definition of Hom-Leibniz color algebra introduced in \cite{Multiplicative-S}. We construct a Hom-Leibniz color algebra starting from a given $n$-Hom-Lie color algebra.
\begin{df}  
A \textbf{ Hom-Leibniz color} algebra is a quadruple $(\mathcal{L},[\cdot,\cdot],\varepsilon,\alpha)$ consisting of a $\Gamma$-graded vector space $ \mathcal{L}$, a bicharacter  $\varepsilon$, a bilinear map of degree zero
$[\cdot,\cdot]:\mathcal{L}\times\mathcal{L}\rightarrow\mathcal{L}$  and a  homomorphism   $\alpha:\mathcal{L}\rightarrow\mathcal{L}$ such that, for any  homogeneous elements $x,y,z\in \mathcal{L}$
\begin{align}
\label{LeibnizId1}[\alpha(x),[y,z]]-\varepsilon(x,y)[\alpha(y),[x,z]]=[[x,y],\alpha(z)]&\ \ (\varepsilon-\textrm{Hom-Leibniz\ identity}).
\end{align}
 In particular, if $\alpha$ is a morphism of  Hom-Leibniz color algebra $($i.e. $\alpha\circ[\cdot,\cdot]=[\cdot,\cdot]\circ\alpha^{\otimes2})$ we call   $(\mathcal{L},[\cdot,\cdot],\varepsilon,\alpha)$ a multiplicative Hom-Leibniz color  algebra.
\end{df}
%Let $(\mathcal{L},[\cdot,\cdot],\alpha)$ and $(\mathcal{L}',[\cdot,\cdot]',\alpha')$ be two Hom-Leibniz color algebras. An   linear map of degree zero $f:\mathcal{L}\rightarrow\mathcal{L}'$ is called \textbf{morphism} of  Hom-Leibniz color algebras if $f$ is a weak morphism and $f\circ\alpha=\alpha'\circ f$.

Let $(\mathfrak{g}, [\cdot,\cdots,\cdot],\varepsilon, \alpha)$ be a $n$-Hom-Lie color algebra, we define  the $\Gamma$-graded space $\mathcal{L}=\mathcal{L}(\mathfrak{g}) := \bigwedge^{n-1} \mathfrak{g} $, which is called fundamental set, we define, for a fundamental object $X = x_1 \wedge \cdots \wedge x_{n-1}$ $\in \mathcal{L}$, an adjoint map $ad_X$ as a linear map on  $\mathfrak{g}$ by :
\begin{equation}\label{adjoint}
  ad_X \cdot y = [x_1,\cdots, x_{n-1},y].
\end{equation}
We define a linear map $\tilde{\alpha} : \mathcal{L} \longrightarrow \mathcal{L}$ by:
$$ \tilde{\alpha}(X) = \alpha(x_1) \wedge \cdots \wedge\alpha(x_{n-1}).$$
Then the color $\varepsilon$-$n$-Hom-Jacobi identity $\eqref{identity}$ may be written in terms of adjoint maps as:
\begin{equation*}
   ad_{\tilde{\alpha}(X)} \cdot [y_1,\cdots,y_n] = \sum_{i=1}^{n} \varepsilon(X,Y_{i} )[\alpha(y_1),\cdots,\alpha(y_{i-1}),  ad_X \cdot  y_i,\cdots,\alpha( y_{n})].
\end{equation*}
Now, we define a  bilinear map of degree zero $[\cdot,\cdot]_\alpha : \mathcal{L} \times \mathcal{L} \longrightarrow \mathcal{L}$ by :
\begin{equation}\label{crochet-leibniz}
    [X,Y]_\alpha
= \sum_{i=1}^{n-1} \varepsilon(X,Y_i )( \alpha(y_1),\cdots,\alpha(y_{i-1}),  ad_X \cdot y_i,\cdots,\alpha( y_{n-1})).\end{equation}
\begin{prop} \label{propo-leibniz}With the above notations,
the map $ad$ satisfies for all $X,Y \in \mathcal{H}(\mathcal{L}(\mathfrak{g})) $ and $z \in \mathcal{H}(\mathfrak{g})$ the equality
\begin{equation}\label{map-ad}
 ad_{[X,Y]_\alpha} \cdot \alpha (z) = ad_{\tilde{\alpha}(X)} \cdot ( ad_{Y} \cdot z)  - \varepsilon (X,Y ) \ ad_{\tilde{\alpha}(Y)} \cdot ( ad_{X} \cdot  z) .   \end{equation}
Moreover,  the quadruple $(\mathcal{L},[\cdot,\cdot]_\alpha,\varepsilon,\tilde{\alpha})$ is a Hom-Leibniz color algebra.
\end{prop}
\begin{proof}
It is easy to show that the identity \eqref{identity} is equivalent to  \eqref{map-ad}. Let $X , Y,Z  \in \mathcal{H}(\mathcal{L})$,  the $\varepsilon$-Hom-Leibniz identity \eqref{LeibnizId1} can be written by the bracket $[\cdot,\cdot]_\alpha$ and the twist $\tilde{\alpha}$ as
\begin{align}\label{LeibnizId2} [\tilde{\alpha}(X),[Y,Z]_\alpha]_\alpha-\varepsilon(X,Y)[\tilde{\alpha}(Y),[X,Z]_\alpha]_\alpha=[[X,Y]_\alpha,\tilde{\alpha} (Z)]_\alpha.
\end{align}
Then, we have
\begin{align*}
   &[\tilde{\alpha}(X),[Y,Z]_\alpha]_\alpha \\
    =&\sum_{i=1}^{n-1} \sum_{j <i}^{n-1}\varepsilon(X,Z_j)\varepsilon(Y,Z_i)  (\alpha^2(z_1),\ldots, \alpha(ad_X z_j),\ldots,\alpha(ad_Y z_i),\ldots,\alpha^2(z_{n-1}))\\
    &+\sum_{i=1}^{n-1} \sum_{j >i}^{n-1}\varepsilon(X,Y+Z_j)\varepsilon(Y,Z_i)  (\alpha^2(z_1),\ldots, \alpha(ad_X z_i),\ldots,\alpha(ad_Y z_j),\ldots,\alpha^2(z_{n-1}))\\
    &+\sum_{i=1}^{n-1} \varepsilon(X+Y,Z_i) (\alpha^2(z_1),\ldots,  ad_{\tilde{\alpha}(X)} \cdot ad_y  z_i),\ldots,\alpha^2(z_{n-1})).
\end{align*}
Using the $\varepsilon$-skew-symmetry in $X$ and $Y$, we obtain
\begin{align*}
   & [\tilde{\alpha}(Y),[X,Z]_\alpha]_\alpha\\
    =&\sum_{i=1}^{n-1} \sum_{j <i}^{n-1}\varepsilon(Y,Z_j)\varepsilon(X,Z_i)  (\alpha^2(z_1),\ldots, \alpha(ad_Y z_j),\ldots,\alpha(ad_X z_i),\ldots,\alpha^2(z_{n-1}))\\
    &+\sum_{i=1}^{n-1} \sum_{j >i}^{n-1}\varepsilon(Y,X)\varepsilon(Y,Z_j)\varepsilon(X,Z_i)  (\alpha^2(z_1),\ldots, \alpha(ad_Y z_i),\ldots,\alpha(ad_X z_j),\ldots,\alpha^2(z_{n-1}))\\
    &+\sum_{i=1}^{n-1} \varepsilon(X+Y,Z_i) (\alpha^2(z_1),\ldots,  ad_{\tilde{\alpha}(Y)} \cdot ad_X z_i),\ldots,\alpha^2(z_{n-1})).
\end{align*}
So, the right hand side of Eq.\eqref{LeibnizId2} is equal to
\begin{align}
    &\sum_{i=1}^{n-1} \varepsilon(X+Y,Z_i) \Big(\alpha^2(z_1),\ldots, ( ad_{\tilde{\alpha}(X)} \cdot ad_y  z_i)-\varepsilon(X,Y) ad_{\tilde{\alpha}(Y)} \cdot ad_X z_i)),\ldots,\alpha^2(z_{n-1})\Big).
\end{align}
The left hand side of Eq.\eqref{LeibnizId2} is equal to
\begin{align*}
    &[[X,Y]_\alpha,\tilde{\alpha} (Z)]_\alpha \\
   = & \sum_{i=1}^{n-1} \sum_{j =1}^{n-1}\varepsilon(X,Y_j)\varepsilon(X,Z_i)\varepsilon(Y,Z_i)  (\alpha^2(z_1),\ldots, [\alpha(y_1),\ldots,ad_X z_j,\ldots,\alpha( y_{n-1}),\alpha(z_i)],\ldots,\alpha^2(z_{n-1}))\\
   =& \sum_{i=1}^{n-1} \varepsilon(X+Y,Z_i) \Big( \alpha^2(z_1),\ldots,\alpha^2(z_{i-1}), ad_{[X,Y]_\alpha} \alpha(z_i),\ldots,  \alpha^2(z_{n-1})        \Big)
\end{align*}
by Eq. \eqref{map-ad},  the $\varepsilon$-Hom-Leibniz identity holds.
\end{proof}

\subsection{Representations of $n$-Hom-Lie color algebras}
In this subsection, we introduce the notion of representation of $n$-Hom-Lie color algebras, which is a generalization of representations of $n$-Hom-Lie  superalgebras $($see in \cite{Cohomology-super}$)$ in the $\Gamma$-graded case. In the sequel, we consider only multiplicative $n$-Hom-Lie color algebras. 
\begin{df}
A \textbf{representation} of a $n$-Hom-Lie   color  algebra    $\mathcal{G}=(\mathfrak{g},[\cdot ,\cdots,\cdot ],\varepsilon,\alpha)$
on a $\Gamma$-graded space $M$ is a  $\varepsilon$-skew-symmetric linear map of degree zero $\rho:\bigwedge^{n-1} \mathfrak{g}\longrightarrow End(M)$, and $\mu$  an   endomorphism on $M$
satisfying  for $X=(x_1,\ldots,x_{n-1}),Y= (y_1,\ldots,y_{n-1})\in \mathcal{H}(\bigwedge^{n-1} \mathfrak{g})$ and $x_n\in \mathcal{H}(\mathfrak{g})$
\small{\begin{align}
   \rho(\tilde{\alpha}(X)) \circ \mu = \mu \circ \rho(X),
\end{align}}\small{
\begin{align} \nonumber
&\rho(\alpha (x_{1}),\cdots, \alpha(x_{n-1}))\circ\rho(y_1,\ldots,y_{n-1})-\varepsilon(X,Y)\rho(\alpha (y_{1}),\cdots, \alpha(y_{n-1}))\circ\rho(x_1,\ldots,x_{n-1}) \\
&\label{RepIdentity1}\quad\quad\quad\quad\quad=\sum_{i=1}^{n-1}\varepsilon(X,Y_i)\rho(\alpha (y_{1}),\cdots,\alpha (y_{i-1}),ad_X(y_i),\cdots, \alpha(y_{n-1}))\circ \mu,
\end{align}}\small{
\begin{align} \nonumber
    &\rho\big{ (}[x_1,x_2,\cdots,x_n],\alpha(y_1),\alpha(y_2),\cdots,\alpha(y_{n-2})\big{)} \circ \mu =  \\
     \label{RepIdentity2}&\quad\quad\quad\quad\quad\quad\quad\sum_{i=1}^{n} (-1)^{n-i}
    \varepsilon(x_i,X^i) \rho(\alpha(x_{1)},\cdots,\hat{x_i},\cdots,\alpha(x_n)) \rho(x_i, y_1,y_2,\cdots,y_{n-2}).
\end{align}}
We denote this representation by the triple $(M,\rho,\mu)$.\end{df}\begin{rems} \
\begin{enumerate}\item  In term of fundamental object  on $\mathcal{L}$ and the bracket of Hom-Leibniz color algebra defined in \eqref{crochet-leibniz}. The condition \eqref{RepIdentity1} can be written as
\begin{align} \label{identity-rep3}
&\rho(\tilde{\alpha}(X))\circ\rho(Y)-\varepsilon(X,Y)\rho(\tilde{\alpha}(Y))\circ\rho(X) =\rho([X,Y]_\alpha)\circ \mu.
\end{align}
  \item
Two representations $(M,\rho,\mu )$ and $(M',\rho',\mu' )$ of a $n$-Hom-Lie color algebra $\mathcal{G}$ are \emph{equivalent} if there exists $f:M \rightarrow M' $, an isomorphism of $\Gamma$-graded vector space of degree zero, such that $f(\rho(X) m)=\rho '(X) f(m)$ and $f\circ \mu =\mu' \circ f$ for $X\in \mathcal{H} (\mathcal{L})$, $m\in M$ and $m'\in M'$.
\item If $\alpha=id_\mathfrak{g}$ and $\mu=id_M$, we recover  representations of $n$-Lie color algebras.
\end{enumerate}
\end{rems}
\begin{exa}Let $(\mathfrak{g},[\cdot ,\cdots,\cdot ],\varepsilon,\alpha)$ be a multiplicative  $n$-Hom-Lie color  algebra. The map $ad$  defined in Eq. \eqref{adjoint} is a representation on $\mathfrak{g}$ where the endomorphism  $\mu$ is the twist map $\alpha$. The identity \eqref{RepIdentity1} is equivalent to $n$-Hom-Jacobi Eq. \eqref{identity}. It is called the \textbf{adjoint representation}.
\end{exa}

\begin{prop}\label{tri}
Let $(\mathfrak{g},[\cdot,\cdots,\cdot],\varepsilon,\alpha)$ be a multiplicative  $n$-Hom-Lie color algebra. Then $(M,\rho,\mu)$ is a representation of $\mathfrak{g}$ if and only if $(\mathfrak{g}\oplus M,[\cdot,\ldots,\cdot]_{\rho},\varepsilon,\tilde{\alpha})$ is a  multiplicative $n$-hom-Lie color algebra with the bracket operation $[\cdot,\ldots,\cdot]_{\rho}:\wedge^n(\mathfrak{g}\oplus M)\longrightarrow\mathfrak{g}\oplus M$ defined by:
$$[x_1+m_1,\ldots,x_n+m_n]_{\rho}=[x_1,\ldots,x_n]+\sum_{i=1}^{n}(-1)^{n-i}  \varepsilon(x_i,X^i)\rho(x_1,\ldots,\widehat{x}_i,\ldots,x_n)m_i$$
and  the  linear map of degree zero $\alpha_{\mathfrak{g}\oplus M}:\mathfrak{g}\oplus M\longrightarrow \mathfrak{g}\oplus M$ given by : $$\alpha_{\mathfrak{g}\oplus M}(x+m)=\alpha(x)+\mu(m)$$
for all  $x_i\in \mathcal{H} (\mathfrak{g})$ and $m_i \in M$, $(i\in \{1,\ldots,n\})$. It is called the { semi-direct product} of the  $n$-Hom-Lie color algebra
$(\mathfrak{g},[\cdot,\cdots,\cdot],\varepsilon,\alpha)$ by the representation $(M,\rho,\mu)$ denoted by $\mathfrak{g}\ltimes_\rho^{\alpha,\mu} M$.
 Note that $\mathfrak{g}\oplus M$ is a $\Gamma$-graded space, where $(\mathfrak{g}\oplus M)_\gamma=\mathfrak{g}_\gamma\oplus M_\gamma$ implying that  $x+m \in \mathcal{H}(\mathfrak{g}\oplus M)$ then $\overline{x_i+m_i}=\overline{x_i}=\overline{m_i}$.
\end{prop}
\begin{proof}
It is easy to show that $[\cdot,\ldots,\cdot]_{\rho}$ is  $\varepsilon$-skew-symmetric using the $\varepsilon$-skew-symmetry of  $[\cdot, \cdots, \cdot]$ and $\rho$. Let $x_1+v_1,\ldots,x_{n-1}+v_{n-1}$ and $y_1+w_1,\ldots,y_{n}+w_{n} \in \mathcal{H}(\mathfrak{g}\oplus M)$,  the identity \eqref{identity} is give in term of $[\cdot,\ldots,\cdot]_{\rho}$ and $\alpha_{\mathfrak{g}\oplus M}$ by
{\small\begin{align}\label{identityrho}
   &[\alpha_{\mathfrak{g}\oplus M}(x_1+v_1) , \ldots, \alpha_{\mathfrak{g}\oplus M}(x_{n-1}+v_{n-1}), [y_1+w_1,\ldots,y_n+w_n]_\rho]_\rho=\\
   \nonumber &\sum_{i=1}^n \varepsilon(X,Y_i)  [\alpha_{\mathfrak{g}\oplus M}(y_1+w_1) , \ldots,\alpha_{\mathfrak{g}\oplus M}(y_{i-1}+w_{i-1}), [x_1+v_1,\ldots,x_{n-1}+v_{n-1},y_i+w_i]_\rho,\ldots,\alpha_{\mathfrak{g}\oplus M}(y_n+w_n) ]_\rho.
\end{align}}
 The left-hand side of \eqref{identityrho} is equal to
\begin{align*}
   & [\alpha(x_1) , \ldots, \alpha(x_{n-1}), [y_1,\ldots,y_n]]\\ +& \sum_{i=1}^{n-1} (-1)^{n-i} \varepsilon(x_i, X^i +Y ) \rho(\alpha(x_1),\ldots,\hat{x_i},\ldots,\alpha(x_{n-1}),[y_1,\ldots,y_n]) \mu (v_i)\\
   +&\sum_{i=1}^n (-1)^{n-i}\varepsilon(y_i,Y^i) \rho(\alpha(x_1),\ldots,\alpha(x_n)) \rho(y_1,\ldots,\hat{y_i},\ldots,y_n)w_i
\end{align*}
The right-hand side of \eqref{identityrho}, for a fixed  $i \in \{1,\ldots,n\}$, is equal to :
\begin{align*}
    &\varepsilon(X,Y_i) [\alpha(y_1),\ldots, \alpha(y_{i-1}),[x_1,\ldots,y_i],\ldots,\alpha(y_n)] \\
    +& \sum_{j< i}(-1)^{n-j}\varepsilon(X,Y_i) \varepsilon(y_j,X+Y^j) \rho(\alpha(y_1),\ldots,\hat{y_j},\ldots,\alpha(y_{i-1}),[x_1,\ldots, x_{n-1},y_i],\ldots,\alpha(y_n)) \mu(w_j)\\
    +&  \sum_{j >i} (-1)^{n-j} \varepsilon(X,Y_i) \varepsilon(y_j, Y^j ) \rho(\alpha(y_1),\cdots,\alpha(y_{i-1}),[x_1,\ldots,x_{n-1},y_i],\ldots,\hat{y_j},\ldots,\alpha(y_n)) \mu(w_j)\\
   +&  \sum_{j=1}^{n-1} (-1)^{i+j}\varepsilon(X,Y_i)  \varepsilon(X+y_i,Y^i) \varepsilon (x_j,X^j+y_i) \rho(\alpha(y_1),\ldots,\hat{y_i},\ldots,\alpha(y_n)) \rho(x_1,\ldots,\hat{x_j},\ldots, x_{n-1},y_i) v_j\\
   + &  (-1)^{n-i}\varepsilon(X,Y_i)  \varepsilon(X+y_i,Y^i) \rho(\alpha(y_1),\ldots,\hat{y_i},\ldots,\alpha(y_n)) \rho(x_1,\ldots,x_{n-1}) w_i.
\end{align*}
Then $(M,\rho,\mu)$ is a representation of the multiplicative  $n$-Hom-Lie color algebra $(\mathfrak{g},[\cdot,\cdots,\cdot],\varepsilon,\alpha)$, if and only if, $(\mathfrak{g}\oplus M,[\cdot,\ldots,\cdot]_{\rho},\varepsilon,\alpha_{\mathfrak{g}\oplus M})$ is a multiplicative $n$-Hom-Lie color algebra.
\end{proof}
\begin{prop}
Let $(\mathfrak{g},[\cdot,\ldots,\cdot],\varepsilon)$ be a  $n$-Lie color algebra, $(M,\rho)$ be a representation of $\mathfrak{g}$, $\alpha : \mathfrak{g } \rightarrow \mathfrak{g}$ be an algebra  morphism and $\mu : M \to M $ be a linear map of degree zero  such that
$
   \rho(\tilde{\alpha}(X)) \circ \mu = \mu \circ \rho(X)
$, for all $X\in \mathcal{L}$.
Then $(M,\tilde{\rho}=\mu\circ\rho,\mu)$ is a representation of the multiplicative $n$-Hom-Lie color algebra  $(\mathfrak{g},[\cdot,\ldots,\cdot]_\alpha,\varepsilon, \alpha)$ given in Corollary \ref{twist}.
\end{prop}
\begin{proof}It is easy to verify that,
for  $X=(x_1,\ldots,x_{n-1}),Y= (y_1,\ldots,y_{n-1})\in \mathcal{H}(\mathcal{L})$ and  $x_n,y_n \in \mathcal{H}(\mathfrak{g})$,
 \begin{align*}
     \tilde{\rho}(\tilde{\alpha}(X)) \circ \mu  = \mu \circ \tilde{\rho}(X).
 \end{align*}
Since $(M,\rho)$ is a representation of $\mathfrak{g}$, then we have
 \begin{align*}
     &\tilde{\rho}([x_1,\ldots,x_n]_\alpha,\alpha(y_{1}),\ldots,\alpha(y_{n-2})) \circ \mu \\
     = & \mu  \circ\tilde{\rho}([x_1,\ldots,x_n],y_{1},\ldots,y_{n-2}) \\
      =& \sum_{i=1}^n (-1)^{n-1} \varepsilon(x_i,X^i ) \mu^2  \circ \rho(x_1,\ldots,\hat{x_i},\ldots,x_n) \rho(x_i,y_1,\ldots,y_n)\\
       =& \sum_{i=1}^n (-1)^{n-1} \varepsilon(x_i,X^i )\mu  \circ \rho(\alpha(x_1),\ldots,\hat{x_i},\ldots,\alpha(x_n)) \mu \circ \rho(x_i,y_1,\ldots,y_n) \\
        =& \sum_{i=1}^n (-1)^{n-1} \varepsilon(x_i,X^i )  \tilde{\rho}(\alpha(x_1),\ldots,\hat{x_i},\ldots,\alpha(x_n)) \tilde{\rho}(x_i,y_1,\ldots,y_n).
     \end{align*}

    Thus,  Condition  \eqref{RepIdentity2} holds. Similarly, Eq. \eqref{RepIdentity1} is valid for $\tilde{\rho}$.
   Then $(M,\tilde{\rho},\mu)$ is a representation of the multiplicative $n$-Hom-Lie color algebra  $(\mathfrak{g},[\cdot,\ldots,\cdot]_\alpha,\varepsilon, \alpha)$.
\end{proof}
\section{Cohomology of  $n$-Hom-Lie  color  algebras}\label{cohomology}
In this Section, we study  a cohomology of multiplicative $n$-Hom-Lie color algebras with respect to a given  representation.
Let $(\mathfrak{g},[\cdot ,\cdots,\cdot ],\varepsilon,\alpha)$ be a    $n$-Hom-Lie color algebra  and $(M,\rho,\mu)$  a $\mathcal{L}(\mathfrak{g})$-modules.
A $p$-cochain is a  $\varepsilon$-skew-symmetric multilinear  map
 $
  \varphi:\underbrace{\mathcal{L}(\mathfrak{g})\otimes \cdots \otimes \mathcal{L}(\mathfrak{g})}_{p-1}\wedge \mathfrak{g} \longrightarrow M$,
such that %\begin{equation}
    %\varphi (X_1,\cdots,X_i,X_{i+1},\cdots , X_{p-1},z) = - \varepsilon (X_i,X_{i+1})\varphi (X_1,\cdots,X_{i+1},X_i,\cdots , X_{p-1},z)\end{equation}
\begin{equation}
\mu\circ\varphi(X_1,\cdots ,X_p,z)=\varphi(\tilde{\alpha}(X_1),\cdots,\tilde{\alpha}(X_p),\alpha(z)).\end{equation}
The space  of all $p$-cochains is $\Gamma$-graded and is denoted by $\mathcal{C}^p_{\alpha,\mu}(\mathfrak{g},M)$.

Thus, we can define the coboundary operator  of the cohomology of a $n$-Hom-Lie color algebra $\mathfrak{g}$ with coefficients in M
by using the structure of its induced Hom-Leibniz color algebra  as
follows.

  \begin{df}We call, for $p\geq 1$, a $p$-coboundary operator of a  $n$-Hom-Lie color  algebra $(\mathfrak{g},[\cdot,\ldots,\cdot ],\varepsilon,\alpha)$, a linear map $\delta^p:\mathcal{C}^p_{\alpha,\mu}(\mathfrak{g},M)\rightarrow \mathcal{C}^{p+1}_{\alpha,\mu}(\mathfrak{g},M)$ defined by:
\begin{align}
                                    &\label{Nambucohomo}\delta^p\varphi(X_1,...,X_p,z) \\=&\sum_{1\leq i< j}^{p}(-1)^i\varepsilon(X_i,X_{i+1}+...+X_{j-1})\varphi\big(\tilde{\alpha}(X_1),...,\widehat{\tilde{\alpha}(X_i)},...,
                                  \tilde{\alpha}(X_{j-1}),[X_i,X_j]_\alpha,...,\tilde{\alpha}(X_{p}),\alpha(z)\big)\nonumber\\
                                    &+ \sum_{i=1}^{p}(-1)^{i}\varepsilon(X_i,X_{i+1}+...+X_{p})\varphi\big(\tilde{\alpha}(X_1),...,\widehat{\tilde{\alpha}(X_i)},...,
                                    \tilde{\alpha}(X_{p}),\textrm{ad}(X_i)(z)\big) \nonumber\\
                                    &+ \sum_{i=1}^{p}(-1)^{i+1}\varepsilon(\varphi+X_{1}+...+X_{i-1},X_i)\rho(\tilde{\alpha}^{p-1}(X_i))\Big( \varphi\big(X_1,...,\widehat{X_i},...,
                                  X_{p},z\big)\Big)\nonumber\\
                                  &+(-1)^{p+1} \big(\varphi(X_1,...,X_{p-1},\ )\cdot X_{p}\big)\bullet_{\alpha}\alpha^p(z)\nonumber,
                                \end{align}
  where
 \begin{align}
 \nonumber&\big(\varphi(X_1,...,X_{p-1},\ )\cdot X_{p}\big)\bullet_{\alpha}\alpha^p(z)\\ \nonumber
  &=\dl\sum_{i=1}^{n-1} (-1)^{n-i}\varepsilon(\varphi+X_{1}+...+X_{p-1},x_{p}^1+...+\hat{x}_{p}^i+\ldots+x_{p}^{n-1}+z ) \varepsilon (x_{p}^i,x_{p}^{i+1}+ \ldots+x_{p}^{n-1}+z)\\
  &\rho(\alpha^{p-1}(x_{p}^1),...,\alpha^{p-1}(x_{p}^{n-1}),\alpha^{p-1}(z))\varphi(X_1,...,X_{p-1},x_{p}^i )
  \end{align}
  for $X_i=(x_i^j)_{1\leq j\leq n-1}\in \mathcal{H}(\mathcal{L}(\mathfrak{g})),\ 1\leq i\leq p$ and $z\in \mathcal{H}(\mathfrak{g})$.
  \end{df}

  \begin{prop}\label{opercob}Let $\varphi\in \mathcal{C}^p_{\alpha,\mu}(\mathfrak{g},M)$ be a $p$-cochain then
  $$\delta^{p+1}\circ\delta^p(\varphi)=0.$$
  \end{prop}
\begin{proof}
Let $\varphi$ be a $p$-cochain, $X_i=(x_i^j)_{1\leq j\leq n-1}\in \mathcal{H}(\mathcal{L}(\mathfrak{g})),\ 1\leq i\leq p+2$ and $z\in \mathcal{H}(\mathfrak{g})$.  We set
 \begin{eqnarray*}
  % \nonumber to remove numbering (before each equation)
       \delta^p=\delta_1^p+\delta_2^p+\delta_3^p+\delta_4^p,
   \quad  \textrm{and} \quad \delta^{p+1}\circ\delta^p= \sum_{i,j=1}^4\Upsilon_{ij},
  \end{eqnarray*}
  where $\Upsilon_{ij} =\delta_i^{p+1}\circ\delta_j^p$ and
 {\small \begin{eqnarray*}
  % \nonumber to remove numbering (before each equation)
    \delta_1^p\varphi(X_1,...,X_{p},z )&=& \sum_{1\leq i< j}^{p}(-1)^{i}\varepsilon(X_{i+1}+...+X_{j-1},X_i)\varphi\big(\tilde{\alpha}(X_1),...,\widehat{\tilde{\alpha}(X_i)},...,
                                  ,[X_i,X_j]_{\alpha},...,\tilde{\alpha}(X_{p}),\alpha(z)\big) \\
  \delta_2^p\varphi(X_1,...,X_{p},z) &=&  \sum_{i=1}^{p}(-1)^{i}\varepsilon(X_{i+1}+...+X_{p},X_i)\varphi\big(\tilde{\alpha}(X_1),...,\widehat{\tilde{\alpha}(X_i)},...,
                                    \tilde{\alpha}(X_{p}),\textrm{ad}(X_i)(z)\big) \\
   \delta_3^p\varphi(X_1,...,X_{p},z) &=& \sum_{i=1}^{p}(-1)^{i+1}\varepsilon(\varphi+X_{1}+...+X_{i-1},X_i)\rho(\tilde{\alpha}^{p-1}(X_i))\Big(   \varphi\big(X_1,...,\widehat{X_i},...,
                                  X_{p},z\big)\Big) \\
    \delta_4^p\varphi(X_1,...,X_{p},z) &=&  (-1)^{p+1} \big(\varphi(X_1,...,X_{p-1},\ )\cdot X_{p}\big)\bullet_{\alpha}\alpha^p(z).
  \end{eqnarray*}}
To simplify the notations we replace $\textrm{ad}(X)(z)$ by $X\cdot z$.
  Let first prove that $\Upsilon_{11}+\Upsilon_{12}+\Upsilon_{21}+\Upsilon_{22}=0$ and     $(X_i)_{1\leq i\leq p}\in \mathcal{L}(\mathfrak{g})$ et $z\in \mathfrak{g}$.\\
Let us compute first $\Upsilon_{11}(\varphi)(X_1,...,X_{p},z)$. We have
\begin{align*}
     & \Upsilon_{11}(\varphi)(X_1,...,X_{p},z)\\
    =& \sum_{1\leq i<k< j}^{p}(-1)^{i+k}\varepsilon(X_{i+1}+...+X_{j-1},X_i)\varepsilon(X_{k+1}+...+X_{j-1},X_k)\\&\quad\quad\quad\quad \varphi\big(\tilde{\alpha}^2(X_1),...,\widehat{X_i},...,\widehat{\tilde{\alpha}(X_k)},...,[\tilde{\alpha}(X_k),[X_i,X_j]_\alpha]_\alpha,....,
\tilde{\alpha}^2(X_{p}),\alpha^2(z)\big) \\
    &+ \sum_{1\leq i<k< j}^{p}(-1)^{i+k-1}\varepsilon(X_{i+1}+...+\widehat{X_{k}}+...+X_{j-1},X_i)\varepsilon(X_{k+1}+...+X_{j-1},X_k)\\& \quad\quad\quad\quad\varphi\big(\tilde{\alpha}^2(X_1),...,\widehat{\tilde{\alpha}(X_i)},...,\widehat{X_k},...,[\tilde{\alpha}(X_i),[X_k,X_j]_\alpha]_\alpha,....,
\tilde{\alpha}^2(X_{p}),\alpha^2(z)\big) \\
    &+ \sum_{1\leq i<k< j}^{p}(-1)^{i+k-1}\varepsilon(X_{i+1}+...+\widehat{X_{k}}+...+X_{j-1},X_i)\varepsilon(X_{k+1}+...+X_{j-1},X_k)\\&\quad\quad\quad\quad \varphi\big(\tilde{\alpha}^2(X_1),...,\widehat{X_i},...,\widehat{[X_i,X_k]_\alpha},...,[[X_i,X_k]_\alpha,\tilde{\alpha}(X_j)]_\alpha,....,
\tilde{\alpha}^2(X_{p}),\alpha^2(z)\big).
\end{align*}
Whence applying the Hom-Leibniz identity \eqref{LeibnizId1} to $X_i,\ X_j,\ X_k\in \mathcal{H}(\mathcal{L}(\mathfrak{g}))$, we find
$\Upsilon_{11}=0$.\\
On the other hand, we have
\begin{align*}
    &  (\Upsilon_{21}(\varphi)+\Upsilon_{12}(\varphi))(X_1,...,X_{p},z)  \\
     =&  \sum_{1\leq i< j}^{p}(-1)^{i+j-1}\varepsilon(X_{i+1}+...+\widehat{X_{j}}+...+X_{p},X_i)\varepsilon(X_{j+1}+...+X_{p},X_j)\\& \quad\quad\quad\quad\varphi\big(\tilde{\alpha}^2(X_1),...,\widehat{X}_i,...,\widehat{[X_i,X_j]_\alpha},...,
\tilde{\alpha}^2(X_{p}),[X_i,X_j]_\alpha\cdot\alpha(z)\big)
\end{align*}
and
\begin{align*}
    &  \Upsilon_{22}(\varphi)(X_1,...,X_{p},z) \\
    =& \sum_{1\leq i< j}^{p}(-1)^{i+j}\varepsilon(X_{i+1}+...+X_{p},X_i)\varepsilon(X_{j+1}+...+X_{p},X_j)\\&\quad\quad\quad\quad \varphi\big(\tilde{\alpha}^2(X_1),...,\widehat{X}_i,...,\widehat{\tilde{\alpha}(X_j)},...,
\tilde{\alpha}^2(X_{p}),\big(\tilde{\alpha}(X_j)\cdot(X_i\cdot z)\big)\big) \\
  &+ \sum_{1\leq i< j}^{p}(-1)^{i+j-1}\varepsilon(X_{j+1}+...+X_{p+1},X_j)\varepsilon(X_{j+1}+...+\widehat{X_{j}}+...+X_{p},X_i)\\&\quad\quad\quad\quad \varphi\big(\tilde{\alpha}^2(X_1),...,\widehat{\tilde{\alpha}(X_i)},...,\widehat{X}_j,...,
\tilde{\alpha}^2(X_{p}),\big(\tilde{\alpha}(X_i)\cdot(X_j\cdot z)\big)\big).
\end{align*}
Then, applying   $\eqref{identity}$  to $X_i,\ X_j\in \mathcal{H}(\mathcal{L}(\mathfrak{g}))$ and $z\in \mathcal{H}(\mathfrak{g})$, $\Upsilon_{12}+\Upsilon_{21}+\Upsilon_{22}=0$.\\
  On the other hand, we have
\begin{align*}
& \Upsilon_{31}\varphi(X_1,...,X_{p+1},z )\\
             &=\sum_{1\leq i<j<k}^{p+1}  \big\{(-1)^{k+i+1}\varepsilon(\varphi+X_{1}+...+X_{k-1},X_k)\varepsilon(X_{i+1}+...+X_{j-1},X_i)\\&\quad\quad\quad\quad\rho\big(\tilde{\alpha}^{p}(X_k)\big) \varphi(\tilde{\alpha}(X_1),...,\widehat{X}_i,...,[X_i,X_j]_\alpha,...,\widehat{X}_k,...,\alpha(z)) \\
&+ (-1)^{j+i+1}\varepsilon(\varphi+X_{1}+...+X_{j-1},X_j)\varepsilon(X_{i+1}+...+\widehat{X_{j}}+...+X_{k-1},X_i)\\&\quad\quad\quad\quad\rho\big(\tilde{\alpha}^{p-1}(X_j)\big) \varphi(\tilde{\alpha}(X_1),...,\widehat{X}_i,...,\widehat{X}_j,...,[X_i,X_k]_\alpha,...,\alpha(z)) \\
                &+ (-1)^{j+i}\varepsilon(\varphi+X_{1}+...+X_{i-1},X_i)\varepsilon(X_{j+1}+...+X_{k-1},X_j)\\&\quad\quad\quad\quad\rho\big(\tilde{\alpha}^{p-1}(X_i)\big)  \varphi(\tilde{\alpha}(X_1),...,\widehat{X}_i,...,\widehat{X}_j,...,[X_j,X_k]_\alpha,...,\alpha(z))\big\},\end{align*}
                \begin{align*}
            &\Upsilon_{13}\varphi(X_1,...,X_{p+1},z ) = -\Upsilon_{31}\varphi(X_1,...,X_{p+1},z )  \\&+ \sum_{1\leq i<j}^{p+1}     (-1)^{i+j+1}\varepsilon(\varphi+X_{1}+...+X_{i-1},X_i)\varepsilon(\varphi+X_{1}+...+\widehat{X_{i}}+...+X_{j-1},X_j)
           \\& \quad\quad\quad\quad\rho\big(\tilde{\alpha}^{p-1}([X_i,X_j]_\alpha)\big) \mu\big(\varphi(X_1,...,\widehat{X}_i,...,\widehat{X}_j,...,z)\big)\end{align*}and
                \begin{align*}
               & \Upsilon_{33}\varphi(X_1,...,X_{p+1},z )\\
             &=\sum_{1\leq i<j}^{p+1} \Big\{  (-1)^{i+j+1}\varepsilon(\varphi+X_{1}+...+X_{i-1},X_i)\varepsilon(\varphi+X_{1}+...+X_{j-1},X_j)\rho\big(\tilde{\alpha}^{p-1}(X_i)\big) \\&\quad\quad\quad\quad \Big(\rho\big(\tilde{\alpha}^p(X_j)\big)\big(\varphi(X_1,...,\widehat{X}_i,...,\widehat{X}_j,...,z)\big)\Big) \\
                &+ (-1)^{i+j}\varepsilon(\varphi+X_{1}+...+X_{i-1},X_i)\varepsilon(\varphi+X_{1}+...+\widehat{X_{i}}+...+X_{j-1},X_j)
                \\&\quad\quad\quad\quad\rho\big(\tilde{\alpha}^{p-1}(X_j)\big)        \Big(\rho\big(\tilde{\alpha}^{p-1}(X_i)\big)\big(\varphi(X_1,...,\widehat{X}_j,...,\widehat{X}_i,...,z)\big)\Big)
            \end{align*}
Then, applying  \eqref{RepIdentity1} to  $\tilde{\alpha}^p(X_i)\in \mathcal{L}(\mathfrak{g})$, $\tilde{\alpha}^p(X_j)\in \mathcal{L}(\mathfrak{g})$ 		and  $\varphi(X_1,...,\widehat{X}_i,...,\widehat{X}_j,...,z)\in M$, we have$$\Upsilon_{13}+\Upsilon_{33}+\Upsilon_{31}=0.$$
By the same calculation, we can prove that
\begin{align*}&\Upsilon_{23}+\Upsilon_{32}=0,\\
&\Upsilon_{14}+\Upsilon_{41}+\Upsilon_{24}+\Upsilon_{42}+\Upsilon_{34}+\Upsilon_{43}+\Upsilon_{44}=0,
\end{align*}
which ends the proof.

\end{proof}
  \begin{df}We define the graded space of
  \begin{itemize}
    \item
   $p$-cocycles  by
  $\mathcal{Z}^p_{\alpha,\mu}(\mathfrak{g},M)=\{\varphi\in \mathcal{C}^p_{\alpha,\mu}(\mathfrak{g},M):\delta^p\varphi=0\},$   \item  $p$-coboundaries  by
  $\mathcal{B}^p_{\alpha,\mu}(\mathfrak{g},M)=\{\psi=\delta^{p-1}\varphi:\varphi\in \mathcal{C}^{p-1}_{\alpha,\mu}(\mathfrak{g},M)\}.$\end{itemize}
  \end{df}
  \begin{lem}$\mathcal{B}^p_{\alpha,\mu}(\mathfrak{g},M)\subset \mathcal{Z}^p_{\alpha,\mu}(\mathfrak{g},M)$.
  \end{lem}
  \begin{df}We call the $p^{\textrm{th}}$-cohomology group  the quotient
  $$\mathcal{H}^p_{\alpha,\mu}(\mathfrak{g},M)=\mathcal{Z}^p_{\alpha,\mu}(\mathfrak{g},M)/\mathcal{B}^p_{\alpha,\mu}(\mathfrak{g},M).$$
  \end{df}
  \begin{rems}
  Let $(\mathfrak{g},[\cdot,\cdots,\cdot],\varepsilon,\alpha)$ be a   $n$-Hom-Lie color algebra, then
  \begin{enumerate}
    \item A $1$-cochain $\phi$ is called $1$-cocycle if and only if ;
    \begin{align} \label{0-cocycle}
       & \phi \circ ad_X (x_n) =  \sum_{i=1}^{n}    (-1)^{n-i} \varepsilon(\phi,\hat{X})
    \varepsilon(x_i,X^i) \rho(x_{1},\cdots,\hat{x_i},\cdots,x_n) \phi(x_i).
     \end{align}
     In particular, if $M=\mathfrak{g}$,  equation  Eq. \eqref{0-cocycle} can be rewritten as \begin{align*}
      \sum_{i=1}^n\varepsilon(\phi,X_i) [x_1,\ldots,\phi (x_i),\ldots,x_n] - \phi ([x_1,x_2,\ldots,x_n]) = 0.
 \end{align*}

    \item A $2$-cochain $\psi$ is called $2$-cocycle, if and only if, for all homogeneous elements $X=(x_1,\ldots,x_{n-1}),Y=(y_1,\ldots,y_{n-1})\in\mathcal{L}(\mathfrak{g})$ and $y_n\in\mathfrak{g}$
   \begin{align} \nonumber
      &\rho (\tilde{\alpha}(X)) \psi(y_1,\dots,y_n)+\psi (\alpha(x_1),\ldots,\alpha(x_{n-1}),ad_{Y} \cdot y_n) \\=\nonumber
      &\sum_{i=1}^n \varepsilon (X,Y_i)\psi (\alpha(y_1),\ldots,\alpha(y_{i-1}),[x_1,\dots,x_{n-1},y_i],\ldots,\alpha(y_n)) \\
    & +\sum_{i=1}^n  (-1)^{n-i}\varepsilon (\psi +X,\hat{Y}_n) \varepsilon(y_i,Y_n^i)\rho (\alpha(y_1),\ldots,\hat{y_i},\ldots,\alpha(y_n)) \psi (x_1,\dots,x_{n-1},y_i).
    \label{2Cocycle}\end{align}
  \end{enumerate}
  \end{rems}
  \section{Deformations of  $n$-Hom-Lie color  algebras}\label{deformations}
  In this Section,  we study formal deformations  and discuss equivalent deformations of  $n$-Hom-Lie color algebras.
  \begin{df}Let $(\mathfrak{g},[\cdot ,\cdots,\cdot ],\varepsilon,\alpha)$ be a  $n$-Hom-Lie color algebra and $\omega_i : \mathfrak{g}^{\otimes n} \to \mathfrak{g}$ be  $\varepsilon$-skew-symmetric linear maps of degree zero. Consider a $\lambda$-parametrized family of $n$-linear operations, ($\lambda \in \mathbb{K}$): \begin{align}\label{deformation}&
  [x_1 ,\cdots,x_n ]_\lambda = [x_1 ,\cdots,x_n ] + \sum_{i = 1 }^\infty  \lambda^i \omega_i(x_1,\ldots,x_n).\end{align}

The tuple    $\mathfrak{g}_\lambda:=(\mathfrak{g},[\cdot,\cdots,\cdot ]_\lambda,\varepsilon,\alpha)$ is   a one-parameter \textbf{formal deformation } of $(\mathfrak{g},[\cdot ,\cdots,\cdot ],\varepsilon,\alpha)$ generated by $ \omega_i$ if it defines a $n$-Hom-Lie color algebra.
 \begin{rems}~~
  \begin{enumerate}
      \item If $\lambda^2=0$, ($k=1$), the deformation is called infinitesimal.
      \item  If $\lambda^n=0$,  the deformation is said to be of order $(n-1)$.
  \end{enumerate}
 \end{rems}Let $\omega$ and $\omega'$ be two $2$-cochains on a  $n$-Hom-Lie color algebra $(\mathfrak{g},[\cdot,\cdots,\cdot ],\varepsilon,\alpha)$ with coefficients in the adjoint representation.
Define the  bracket $[~,~]:\mathcal{C}^2_{\alpha,\alpha}(\mathfrak{g},\mathfrak{g})\times \mathcal{C}^2_{\alpha,\alpha}(\mathfrak{g},\mathfrak{g})\to \mathcal{C}^3_{\alpha,\alpha}(\mathfrak{g},\mathfrak{g}),$ for $X,Y\in\mathcal{H}(\mathcal{L})$ and $z\in \mathcal{H}(\mathfrak{g})$, by
 \begin{align*}
     [\omega,\omega'] (X ,Y,z) =&
     \omega (\tilde{\alpha}(X), \omega'(Y,z)) - \varepsilon(X,Y) \omega (\tilde{\alpha}(Y), \omega'(X,z)) - \omega ( \omega'(X,\bullet) \circ \tilde{\alpha}(Y), \alpha(z)) \\
     &+\omega' (\tilde{\alpha}(X), \omega(Y,z)) - \varepsilon(X,Y) \omega' (\tilde{\alpha}(Y), \omega(X,z)) - \omega' ( \omega(X ,\bullet) \circ \tilde{\alpha}(Y),\alpha( z)),
 \end{align*}
 where \begin{align*}
     &\omega(X,\bullet) \circ \tilde{\alpha}(Y) = \sum_{k=1}^{n-1} \varepsilon (X,Y_i) \alpha(y_1) \wedge \ldots \wedge \omega(X,y_k) \wedge \ldots \wedge \alpha(y_{n-1})
 \end{align*}
for all $X,Y \in \mathcal{H}(\mathcal{L}(\mathfrak{g}))$ and $z \in \mathcal{H}(\mathfrak{g}) $.

 \begin{thm} \label{thm-defo}
  With the above notations, the   $2$-cochains $\omega_i$, $  i\geq 1$,  generate a  one-parameter formal deformation $\mathfrak{g}_\lambda$ of order $k$  of a $n$-Hom-Lie color algebra $(\mathfrak{g},[\cdot ,\cdots,\cdot ],\varepsilon,\alpha)$ if and only if the following conditions hold :
 \begin{equation}\label{deformation1}
     \delta^2 \omega_1 = 0,
 \end{equation}
 \begin{equation}\label{deformation2}
     \delta^2 \omega_l + \frac{1}{2} \ \sum_{i=1}^{l-1} \ [\omega_i,\omega_{l-i}] = 0 ,~~~~~~ 2\leq l\leq k,
 \end{equation}
 \begin{equation}\label{deformation3}
     ~~~~~~ \frac{1}{2} \ \sum_{i=l-k}^{k} \ [\omega_i,\omega_{l-i}] = 0, ~~~~~~ n\leq l\leq 2k.
 \end{equation}
 \end{thm}
 \begin{proof}
 Let  $\omega_i$, $  i\geq 1$ be  $2$-cochains generating  a one-parameter formal  deformation $\mathfrak{g}_\lambda:=(\mathfrak{g},[\cdot,\cdots,\cdot ]_\lambda,\varepsilon,\alpha)$ of order $k$ of a $n$-Hom-Lie color algebra $(\mathfrak{g},[\cdot,\ldots,\cdot],\varepsilon,\alpha)$. Then  $\mathfrak{g}_\lambda$ is also a $n$-Hom-Lie color algebra. According to Proposition \ref{propo-leibniz}, the identity of $\varepsilon$-$n$-Hom-Jacobi identity \eqref{identity} on $\mathfrak{g}_\lambda$  is equivalent to \begin{equation} \label{eq-associee}
 ad_{[X,Y]_\alpha^\lambda} ^\lambda\cdot \alpha (z) = ad_{\tilde{\alpha}(X)} ^\lambda\cdot ( ad_{Y} ^\lambda\cdot z)  - \varepsilon (X,Y ) \ ad_{\tilde{\alpha}(Y)} ^\lambda\cdot ( ad_{X} ^\lambda\cdot z),    \end{equation}
 where$$
     [X,Y]_\alpha^\lambda = [X,Y]_\alpha + \sum_{i = 1}^k\lambda^i \omega_i(X ,\bullet) \circ \tilde{\alpha}(Y)
$$
and  the adjoint map on $\mathfrak{g}_\lambda$ is given by $$ad_{X}^\lambda\cdot z= ad_X \cdot z + \sum_{i =1}^k\lambda^i \omega_i(X,z). $$
 The left-hand side of \eqref{eq-associee} is equal to :
 \begin{align*}
    ad_{[X,Y]_\alpha^\lambda} ^\lambda\cdot\alpha (z)= &ad_{[X,Y]_\alpha} \cdot \alpha(z) + \sum_{i = 1}^k\lambda^i \big{(}  \omega_i([X,Y]_\alpha,\alpha(z))+ ad_{\omega_i(X,\bullet) \circ \tilde{\alpha}(Y)} \cdot \alpha(z) \big{)} \\
    &+ \sum_{i,j = 1}^k\lambda^{i+j} ad_{\omega_i(\omega_j(X,\bullet) \circ \tilde{\alpha}(Y))} \cdot \alpha(z).
   \end{align*}
 The right-hand side of \eqref{eq-associee} is equal to :
 \begin{align*}
     ad_{\tilde{\alpha}(X)} ^\lambda\cdot ( ad_{Y} ^\lambda\cdot z) = & ad_{\tilde{\alpha}(X)} \cdot (ad_Y \cdot z) + \sum_{i = 1}^k\lambda^i \big{(}  \omega_i(\tilde{\alpha}(X),ad_Y \cdot z)+ ad_{\tilde{\alpha}(X)} \cdot \omega_i (Y,z)\big{)} \\
     &+ \sum_{i,j =1}^k\lambda^{i+j} \omega_i(\tilde{\alpha}(X),\omega_j(Y,z))
 \end{align*}
 and
 \begin{align*}
     ad_{\tilde{\alpha}(Y)} ^\lambda\cdot ( ad_{X} ^\lambda\cdot z) = & ad_{\tilde{\alpha}(Y)} \cdot (ad_X \cdot z) + \sum_{i = 1}^k\lambda^i \big{(}  \omega_i(\tilde{\alpha}(Y),ad_X \cdot z)+ ad_{\tilde{\alpha}(Y)} \cdot \omega_i (X,z)\big{)} \\
     &+ \sum_{i,j = 1}^k\lambda^{i+j} \omega_i(\tilde{\alpha}(Y),\omega_j(X,z)).
 \end{align*}
 Comparing the coefficients of $\lambda^l$, we obtain conditions \eqref{deformation1}, \eqref{deformation2} and \eqref{deformation3} respectively.
 \end{proof}
 \begin{rems}~~\begin{enumerate}
     \item
Equation \eqref{deformation1} means that  $\omega_1$ is always  a $2$-cocycle on $\mathfrak{g}$.
\item  If $\mathfrak{g}_\lambda$ is a deformation of  order $k$. Then,   by Eq. \eqref{deformation3}, we deduce that   $(\mathfrak{g},\omega_{k} ,\varepsilon,\alpha)$ is  a $n$-Hom-Lie color algebra.
 \item In particular,
   consider an  infinitesimal deformation of $\mathfrak{g}_\lambda$ generated by  $\omega:\wedge^n\mathfrak{g}\rightarrow\mathfrak{g}$ defined as 
  $$[\cdot ,\cdots\,\cdot ]_\lambda=[\cdot ,\cdots,\cdot ]+\lambda\omega(\cdot ,\cdots,\cdot).$$
  The linear map
$\omega$ generates an infinitesimal deformation of the  multiplicative $n$-Hom-Lie color algebra
$\mathfrak{g}$ if and only if:

  (a) $(\mathfrak{g},\omega,\varepsilon,\alpha)$  is a multiplicative $n$-Hom-Lie color algebra.

  (b)
$\omega$ is a $2$-cocycle of $\mathfrak{g}$ with coefficients in the adjoint representation. That is $\omega$ satisfies condition \eqref{2Cocycle} for $\rho=ad$.
\end{enumerate}
\end{rems}
  \end{df}
 \begin{df}
 Two formal deformations $\mathfrak{g}_\lambda$ and $\mathfrak{g}_{\lambda^{'}}$ of a $n$-Hom-Lie color algebra $(\mathfrak{g},[\cdot,\ldots,\cdot],\varepsilon,\alpha)$ are said to be  equivalent  if there exists a formal  isomorphism  $\phi_\lambda :\mathfrak{g}_\lambda \rightarrow \mathfrak{g}_{\lambda^{'}}$, where $\phi_\lambda = \sum_{i \geq 0} \phi_i \lambda^i $, and, $\phi_i :\mathfrak{g}\rightarrow\mathfrak{g}$ are  linear maps of degree zero  such that $\phi_0 = id_\mathfrak{g}$, $\phi_i \circ \alpha = \alpha \circ \phi_i $ and,
 \begin{align}\label{eq.equivalence}
     \phi_\lambda \circ [x_1,\cdots,x_n]_\lambda = [\phi_\lambda(x_1),\cdots,\phi_\lambda(x_n)]_{\lambda^{'}}.\end{align}
It   is denoted by $\mathfrak{g}_\lambda \sim \mathfrak{g}_{\lambda^{'}} $. A    formal
deformation $\mathfrak{g}_\lambda  $ is said to be  trivial if $\mathfrak{g}_\lambda \sim \mathfrak{g}_{0} $.
 \end{df}
 \begin{thm}
 Let $\mathfrak{g}_\lambda $ and $\mathfrak{g}_{\lambda^{'}}$ be two equivalent  deformations of a $n$-Hom-Lie color algebra $(\mathfrak{g},[\cdot,\ldots,\cdot],\varepsilon,\alpha)$ generated by $\omega$ and $\omega'$ respectively. Then   $\omega$ and $\omega'$ belong to the same cohomology class in the cohomology group $\mathcal{H}^2_{\alpha,\mu}(\mathfrak{g},\mathfrak{g})$
 \end{thm}
 \begin{proof} Its enough to prove that $\omega-\omega' \in B^2(\mathfrak{g},\mathfrak{g})$. \\
 We have  two equivalent deformations $\mathfrak{g}_\lambda $ and $\mathfrak{g}_{\lambda^{'}} $, then identification of coefficients of $\lambda$ in  \eqref{eq.equivalence} leads to:
 \begin{align*}
    & \omega(x_1,\cdots,x_n) + \phi_1 [x_1,\cdots,x_n] =  \omega^{'}(x_1,\cdots,x_n) + [\phi(x_1),\cdots,x_n] + \ldots + [x_1,\cdots,\phi(x_n)].
 \end{align*}
 Thus,
 \begin{align*}
     \omega(x_1,\cdots,x_n) - \omega^{'}(x_1,\cdots,x_n)  =& -\phi [x_1,\cdots,x_n]  +  [\phi(x_1),\cdots,x_n] + \ldots + [x_1,\cdots,\phi(x_n)]\\
      =&-\phi [x_1,\cdots,x_n] + \sum_{i=1}^n (-1)^{n-i} \varepsilon(x_i,X^i) [x_1,\ldots,\hat{x_i},\ldots,x_n,\phi(x_i)]\\
      =& \delta^1 \phi (X,x_n).
 \end{align*}
 Therefore $\omega-\omega^{'} \in B^2(\mathfrak{g},\mathfrak{g})$.
\end{proof}
\section{Nijenhuis operators on $n$-Hom-Lie color algebras}\label{nijenhuis}
Motivated by the infinitesimally  trivial deformation    introduced in this Section, we define the notion of Nijenhuis operator for a multiplicative $n$-Hom-Lie color algebras which is a generalization of Nijenhuis operator on $n$-Lie algebras given in \cite{NijenhuisOperators}.   Then  we define the notion of a product structure on a $n$-Hom-Lie color algebras using  Nijenhuis operators.
 \subsection{Definitions and Constructions}
 Let $(\mathfrak{g},[\cdot,\ldots,\cdot],\varepsilon,\alpha)$ be a  $n$-Hom-Lie color algebra and $\mathfrak{g}_\lambda:=(\mathfrak{g},[\cdot,\ldots,\cdot]_\lambda,\varepsilon,\alpha)$   be a    deformation of $\mathfrak{g}$ of order $(n-1)$.
 \begin{df}
   The  deformation $\mathfrak{g}_\lambda$ is said to be \textbf{infinitesimally  trivial} if there exists a linear  map of degree zero  $\mathcal{N} : \mathfrak{g} \to \mathfrak{g}$  such that $\mathcal{T}_\lambda = id + \lambda \mathcal{N} : \mathfrak{g}_\lambda \to \mathfrak{g} $ is an algebra morphism,  that is,  for all   $x_1,\ldots,x_n \in \mathcal{H}(\mathfrak{g})$, we have
   \begin{align} \label{Nijenhuis-alpha1}
     \mathcal{T}_\lambda  \circ \alpha &= \alpha \circ \mathcal{T}_\lambda,
 \\\label{eq-trivial}
     \mathcal{T}_\lambda[x_1,\ldots,x_n]_\lambda& = [\mathcal{T}_\lambda(x_1),\ldots,\mathcal{T}_\lambda(x_n)].
 \end{align}
   \end{df}
 The condition \eqref{Nijenhuis-alpha1} is equivalent to: \begin{equation} \label{Nijenhuis-alpha2}
     \mathcal{N} \circ \alpha = \alpha \circ \mathcal{N}.
 \end{equation}
 The left hand side of Eq. $\eqref{eq-trivial}$  equals to:
 \begin{align*}
 & [x_1,\ldots,x_n]+ \lambda \big{(} \omega_1(x_1,\ldots,x_n) + \mathcal{N}[x_1,\ldots,x_n] \big{)}\\
  & +\sum_{j=1}^{n-2}  \lambda^{j+1} \big{(}   \mathcal{N} \omega_j(x_1,\ldots,x_n) + \omega_{j+1}(x_1,\ldots,x_n) \big{)} +  \lambda^{n} \mathcal{N} \omega_{n-1}(x_1,\ldots,x_n).
 \end{align*}
The right hand side of Eq. $\eqref{eq-trivial}$  equals to:
\begin{align*}
   & [x_1,\ldots,x_n] + \lambda \sum_{i_1=1}^n [x_1,\ldots,\mathcal{N}x_{i_1},\ldots,x_n] + \lambda^2 \sum_{i_1 < i_2 }^n  [x_1,\ldots,\mathcal{N}x_{i_1},\ldots,\mathcal{N}x_{i_2},\ldots,x_n] + \ldots\\& +  \lambda^{n-1} \sum_{i_1 < i_2 \ldots < i_{n-1}} [x_1,\ldots,\mathcal{N}x_{i_1},\ldots,\mathcal{N}x_{i_2},\ldots,\mathcal{N}x_{i_{n-1}},x_n] +
   \lambda^{n}[\mathcal{N}x_{1},\ldots,\mathcal{N}x_{2},\ldots,\mathcal{N}x_{n}] .
\end{align*}
Therefore, by identification of coefficients, we have
\begin{align}
    \omega_1(x_1,\ldots,x_n) + \mathcal{N}[x_1,\ldots,x_n]& =  \sum_{i_1=1}^n [x_1,\ldots,\mathcal{N}x_{i_1},\ldots,x_n],
\\
    \mathcal{N} \omega_{n-1}(x_1,\ldots,x_n) &= [\mathcal{N}x_{1},\ldots,\mathcal{N}x_{2},\ldots,\mathcal{N}x_{n}],
\\
\mathcal{N} \omega_l(x_1,\ldots,x_n) + \omega_{l-1}(x_1,\ldots,x_n) &= \sum_{i_1 < i_2 \ldots <i_{l}}  [x_1,\ldots,\mathcal{N}x_{i_1},\ldots,\mathcal{N}x_{i_2},\ldots,\mathcal{N}x_{i_{l}},\ldots,x_n]
 \end{align}
 for all $2 \leq l \leq n-1 $.\\  Let $(\mathfrak{g},[\cdot,\ldots,\cdot],\varepsilon,\alpha)$ be a $n$-Hom-Lie color algebra, and $\mathcal{N}: \mathfrak{g} \to \mathfrak{g}$ be a linear map of degree zero. Define a $n$-ary bracket $[\cdot,\ldots,\cdot]_{\mathcal{N}}^1 : \wedge^n \mathfrak{g} \to \mathfrak{g}$ by : \begin{align*}
     [x_1,\ldots,x_n]_{\mathcal{N}}^1 = \sum_{i=1}^n [x_1,\ldots,\mathcal{N}x_i,\ldots,x_n] - \mathcal{N}[x_1,x_2,\ldots,x_n]
 \end{align*}
By induction,  we define  $n$-ary brackets $[\cdot,\ldots,\cdot]_{\mathcal{N}}^j  : \wedge^n \mathfrak{g} \to \mathfrak{g}$, $2 \leq j \leq n-1 $, 
 \begin{align*}
  [x_1,\ldots,x_n]_{\mathcal{N}}^j=  \sum_{i_1 < i_2 \ldots <i_{j}}  [\ldots,\mathcal{N}x_{i_1},\ldots,\mathcal{N}x_{i_2},\ldots,\mathcal{N}x_{i_{j}},\ldots] - \mathcal{N}[x_1,\ldots,x_n]_{\mathcal{N}}^{j-1}.
    \end{align*}
    In particular, we have
    \begin{align}\label{(n-1 )bracket}
        [x_1,\ldots,x_n]_{\mathcal{N}}^{n-1}=  \sum_{i_1 < i_2 \ldots <i_{n-1}}  [\mathcal{N}x_{i_1},\ldots,\mathcal{N}x_{i_{n-1}},x_n] -\mathcal{N}[x_1,\ldots,x_n]_{\mathcal{N}}^{n-2}.
    \end{align}
    
    These observations motivate the following definition.
    \begin{df}
   Let $(\mathfrak{g},[\cdot,\ldots,\cdot],\varepsilon,\alpha)$ be a $n$-Hom-Lie color algebra. A linear map of degree zero $\mathcal{N}: \mathfrak{g} \to \mathfrak{g}$ is called a \textbf{Nijenhuis operator} if it  satisfies $\mathcal{N} \circ \alpha = \alpha \circ \mathcal{N}$ and
   \begin{align} \label{Nijenhuis}
    [\mathcal{N}x_{1},\ldots,\mathcal{N}x_{2},\ldots,\mathcal{N}x_{n}]   =  \mathcal{N} [x_1,\ldots,x_n]_{\mathcal{N}}^{n-1}.
   \end{align}
    \end{df}
    The above condition can be written as :
    {\small\begin{align} \label{Nijenhuis1}
      \sum_{j=0}^{n}(-1)^{n-j}\mathcal{N}^{n-j}\bigg(   \sum_{i_1 < i_2 \ldots <i_{j}}  [\ldots,\mathcal{N}x_{i_1},\ldots,\mathcal{N}x_{i_2},\ldots,\mathcal{N}x_{i_{j}},\ldots]\bigg)=0.
   \end{align}}

    We have seen that any trivial deformation produces a Nijenhuis operator. Conversely,
any Nijenhuis operator gives a trivial deformation as the following theorem shows.
 \begin{thm} Let $(\mathfrak{g},[\cdot,\ldots,\cdot],\varepsilon,\alpha)$ be a $n$-Hom-Lie color algebra and
 $\mathcal{N}$ be a Nijenhuis operator on  $(\mathfrak{g},[\cdot,\cdots,\cdot],\varepsilon,\alpha)$. Then the bracket 
 \begin{align}&
  [x_1 ,\cdots,x_1 ]_\lambda = [x_1 ,\cdots,x_1 ] + \sum_{i= 1 }^{n-1}  \lambda^i [x_1,\ldots,x_n]_\mathcal{N}^i\end{align}
 defines  a deformation of $\mathfrak{g}$ which is  infinitesimally  trivial.
 \end{thm}

 \begin{proof}
  Follows from the above characterization of identity \eqref{eq-trivial}, Theorem \ref{thm-defo} and Lemma \ref{lemme}. \end{proof}
    \begin{prop}
 Let $\mathcal{N}$ be a bijective Nijenhuis operator on a $n$-Hom-Lie color algebra $ (\mathfrak{g},[\cdot,\ldots,\cdot],\varepsilon,\alpha)$, $a\in \mathfrak{g}_0$ such that $\alpha(a)=a$ and $\mathcal{N}(a) \in Z(\mathfrak{g})$. Then $\mathcal{N}$ is a Nijenhuis operator on the  $(n-1)$-Hom-Lie color algebra  $(\mathfrak{g},  \{\cdot,\cdots,\cdot\},\varepsilon, \alpha)$. 
   \end{prop}  Recall that, if $(\mathfrak{g},[\cdot,\cdot],\varepsilon,\alpha)$
be a Hom-Lie color algebra. The Nijenhuis operator condition is written as
\begin{align*}
    [\mathcal{N}x,\mathcal{N}y]= \mathcal{N}[\mathcal{N}x,y]+ \mathcal{N}[x,\mathcal{N}y]-\mathcal{N}^2[x,y].
\end{align*}
\begin{cor}
 Let $\mathcal{N}$ be a Nijenhuis operator on $3$-Hom-Lie color algebra $ (\mathfrak{g},[\cdot,\cdot,\cdot],\varepsilon,\alpha)$. if $\mathcal{N}$ is  a bijection, then it is a Nijenhuis operator on the  Hom-Lie color algebra  $(\mathfrak{g},  \{\cdot,\cdot\},\varepsilon, \alpha)$ such that $\mathcal{N}(a) \in Z(\mathfrak{g})$.
   \end{cor}
\subsection{Product structures  on $n$-Hom-Lie color algebras}
In this subsection, we study the notion of product structure on a $n$-Hom-Lie color algebra and show that it leads to a  special decomposition of the original $n$-Hom-Lie color algebra. Moreover, we introduce
the notion of  strict product structure on a $n$-Hom-Lie color algebra and provide  example.
\begin{df}
Let $\mathcal{G}=(\mathfrak{g},[\cdot,\cdots,\cdot],\varepsilon,\alpha)$ be a $n$-Hom-Lie color algebra.

An \textbf{almost product structure} on $\mathcal{G}$ is a linear map of degree zero $\mathcal{P}  : \mathfrak{g} \to \mathfrak{g}$, $( \mathcal{P}   \neq \pm  Id_\mathfrak{g} )$, satisfying $\mathcal{P}  ^2 =Id_\mathfrak{g}$.

 An almost product structure is called \textbf{product stucture } on $\mathcal{G} $ if it is a Nijenhuis operator.
\end{df}
\begin{rem}
One can understand a product structure on $\mathcal{G} $
as a linear map of degree zero $\mathcal{P} : \mathfrak{g}\to \mathfrak{g}$ satisfying
\begin{align}
 \nonumber   &\mathcal{P}  ^2 = Id,\ \ \ \ \mathcal{P} \alpha=\alpha \mathcal{P},\\ [\mathcal{P}x_{1},\mathcal{P} x_{2},\ldots,\mathcal{P}x_{n}]  & \label{Product-Eq} =\sum_{j=0}^{n-1}(-1)^{n-j-1}\mathcal{P}^{\mu_{n-j}}(   \sum_{i_1 < i_2 \ldots <i_{j}}^n[\ldots,\mathcal{P} x_{i_1},\ldots,\mathcal{P} x_{i_{j}},\ldots])
\end{align}
for $x_1,...,x_n\in \mathfrak{g}$ and
 $\mu_k=\begin{cases}
  1~~~~ if ~~~~ k ~is ~ odd  \\
 0~~~~ if ~~~~ k ~is ~ even.
  \end{cases}$
\end{rem}
\begin{thm}\label{thm-decomposition}
Let $\mathcal{G}=(\mathfrak{g},[\cdot,\cdots,\cdot],\varepsilon,\alpha)$ be a $n$-Hom-Lie color algebra. Then $\mathcal{G}$ has a product structure if and
only if $\mathfrak{g}=\bigoplus_{\gamma\in \Gamma}\mathfrak{g}_{\gamma}$ admits a decomposition 
$
    \mathfrak{g}_\gamma= \mathfrak{g}_{\gamma_{+}} \oplus \mathfrak{g}_{\gamma_{-}},
$
where the eigenspaces $\mathfrak{g}_+=\bigoplus_{\gamma\in \Gamma}\mathfrak{g}_{\gamma_+}$ and $\mathfrak{g}_-=\bigoplus_{\gamma\in \Gamma}\mathfrak{g}_{\gamma_-}$  of $\mathfrak{g}$ associated to the eigenvalues $1$ and $-1$ respectively are subalgebras.
\end{thm}
\begin{proof}
Let $\mathcal{P} $ be a  product structure on   $(\mathfrak{g},[\cdot,\cdots,\cdot],\varepsilon,\alpha)$. According to Eq. \eqref{Product-Eq},  for all  element $x_i \in \mathfrak{g}_{+}$  we have:
\begin{align} \nonumber
  [x_1,...,x_n]&= [\mathcal{P}x_{1},\ldots,\mathcal{P}x_{n}] \\ \nonumber & =\sum_{j=0}^{n-1}(-1)^{n-j-1}\mathcal{P}^{\mu_{n-j}}(   \sum_{i_1 < i_2 \ldots <i_{j}}  [\ldots,\mathcal{P}x_{i_1},\ldots,\mathcal{P}x_{i_2},\ldots,\mathcal{P}x_{i_{j}},\ldots])\\
  &=\sum_{j=0}^{n-1}(-1)^{n-j-1} \binom{n}{j}\mathcal{P}^{\mu_{n-j}} [x_1,...,x_n].
\end{align}
Then, we have :
\begin{align}\label{EqEven}
   [x_1,...,x_n]=  &\sum_{2j+1 < n } \binom{n}{2j+1} \mathcal{P} [x_1,...,x_n] -  \sum_{2j  < n } \binom{n}{2j}    [x_1,...,x_n],\quad \text{if}\ n\ \text{is\ even},\\\label{Eqodd}
  [x_1,...,x_n]=    &\sum_{2j  < n } \binom{n}{2j } \mathcal{P} [x_1,...,x_n] -  \sum_{2j+1  < n } \binom{n}{2j+1}     [x_1,...,x_n],\quad \text{if}\ n\ \text{is\ odd}.
\end{align}
By the binomial theorem, we have
\begin{align}\label{BinN}
    &   \sum_{2j  < n } \binom{n}{2j }  - \sum_{2j +1 < n } \binom{n}{2j+1 } = (-1)^{n+1}.
\end{align}
Apply the above condition to Eqs. \eqref{EqEven} and \eqref{Eqodd}, we get, for all $ x_i \in \mathfrak{g}_{+}$, $\mathcal{P}  [x_1,\ldots,x_n]= [x_1,\ldots,x_n]$. Let $ x \in \mathfrak{g}_{+}$, then we have $\mathcal{P}   \circ \alpha (x)= \alpha \circ \mathcal{P}   (x)= \alpha (x)$, which implies that $\alpha(x)\subseteq \mathfrak{g}_{+} $.  So $\mathfrak{g}_{+}$ is subalgebra of $\mathfrak{g}$. Similarly, we  show that $\mathfrak{g}_{-}$ is subalgebra of $\mathfrak{g}$.

Conversely, we define a linear map of degree zero $\mathcal{P}   : \mathfrak{g} \to \mathfrak{g}$
such that for all $  ~x \in \mathfrak{g}_{+} ~and ~ y \in \mathfrak{g}_{-}$ \begin{equation}\label{eq-product}
    \mathcal{P}   (x+y)= x-y .
\end{equation}
We have $\mathcal{P}^2 (x+y)= \mathcal{P}(x-y)=x+y$, then $\mathcal{P}  ^2= Id$. Since  $  ~x \in \mathfrak{g}_{+}$, then $  \alpha(x )\in \mathfrak{g}_{+}$. Thus $\mathcal{P} \circ \alpha (x) = \alpha(x) = \alpha \circ \mathcal{P} (x)$. Similarly,  $\mathcal{P} \circ \alpha (y) = \alpha \circ \mathcal{P} (y)$.\\
 If  $n$  is even, since  $\mathfrak{g}_{+}$ is a subalgebra of $\mathfrak{g}$ then for $ x_i \in \mathfrak{g}_{+}$,  we have:
\begin{eqnarray*}
&&\sum_{j=0}^{n-1}(-1)^{n-j-1}\mathcal{P}^{\mu_{n-j}}(   \sum_{i_1 < i_2 \ldots <i_{j}}  [\ldots\mathcal{P}x_{i_1},\ldots,\mathcal{P}x_{i_2},\ldots,\mathcal{P}x_{i_{j}},\ldots])\\
  &=& \sum_{2j+1 < n } \binom{n}{2j+1} \mathcal{P} [x_1,...,x_n] -  \sum_{2j  < n } \binom{n}{2j}    [x_1,...,x_n] \\
  & =  &\sum_{2j+1 < n } \binom{n}{2j+1} [x_1,...,x_n] -  \sum_{2j  < n } \binom{n}{2j}    [x_1,...,x_n] \\
  &  =& (\sum_{2j+1 < n } \binom{n}{2j+1} -  \sum_{2j  < n } \binom{n}{2j}    )[x_1,...,x_n] \\
   & \overset{\eqref{BinN}}{=}&[\mathcal{P}x_1,\mathcal{P}x_2,...,\mathcal{P}x_n].
\end{eqnarray*}
Also if $n$ is odd, we have :
\begin{eqnarray*}
&&\sum_{j=0}^{n-1}(-1)^{n-j-1}\mathcal{P}^{\mu_{n-j}}(   \sum_{i_1 < i_2 \ldots <i_{j}}  [\ldots,\mathcal{P}x_{i_1},\ldots,\mathcal{P}x_{i_2},\ldots,\mathcal{P}x_{i_{j}},\ldots])\\&=& (\sum_{2j  < n } \binom{n}{2j } -  \sum_{2j+1  < n } \binom{n}{2j+1}    )[x_1,...,x_n] \\
    &\overset{\eqref{BinN}}{=}&[\mathcal{P}x_1,\mathcal{P}x_2,...,\mathcal{P}x_n].
\end{eqnarray*}
  One may check similarly  for all  $ x_i \in  \mathfrak{g}_{-}$. Then $\mathcal{P}  $ is a product structure on $\mathcal{G}$.
\end{proof}
Let $\mathcal{G}=(\mathfrak{g},[\cdot,\cdots,\cdot],\varepsilon,\alpha)$ be a $n$-Hom-Lie color algebra  and $\Theta : \mathfrak{g} \to \mathfrak{g}$ be a linear map of degree $\gamma$. Then $\Theta$ is said in the  \textbf{centroid} of $\mathcal{G}$ if  for all homogeneous elements $x_i \in \mathfrak{g}$, $\Theta\circ\alpha=\alpha\circ\Theta$ and  \begin{align} \label{centroid}
    \Theta [x_1,x_2,\ldots,x_n] = [\Theta x_1,x_2,\ldots,x_n].
\end{align}
The above identity is equivalent to
\begin{align} \label{centroid1}
    \Theta [x_1,x_2\ldots,x_n] = \varepsilon(\gamma,X_{i})[x_1, \ldots, \underbrace{\Theta x_i}_{i\text{-th place}}, \ldots, x_n].
\end{align}
\begin{df}
An almost product structure $\mathcal{P}  $  on  $\mathcal{G}$ is called   a \textbf{strict product structure} if it is an element  of the centroid.\end{df}

\begin{lem} \label{Lemme-strict}
 Let $\mathcal{P}$ be a strict product structure on $\mathcal{G}$, then $\mathcal{P}$ is a product structure on  $\mathcal{G}$ such that $[\mathfrak{g}_{+}, \ldots, \mathfrak{g}_{+}, \underbrace{\mathfrak{g}_{+}}_{i\text{-th place}}, \mathfrak{g}_{-}, \ldots, \mathfrak{g}_{-}] =0 $,  for all $ 1\leq i \leq n-1$.
\end{lem}
\begin{proof} The identity \eqref{Product-Eq} is equivalent to:
\begin{align*} \label{Product-Eq1}
    &\sum_{j=0}^{n}(-1)^{n-j}\mathcal{P}^{\mu_{n-j}}\bigg(   \sum_{i_1 < i_2 \ldots <i_{j}}^n  [\ldots\mathcal{P}x_{i_1},\ldots,\mathcal{P}x_{i_2},\ldots\mathcal{P}x_{i_{j}},\ldots]\bigg)=0.
\end{align*}
Then, if  $\mathcal{P}$ is a strict product structure on $\mathcal{G}$ and $x_i \in \mathfrak{g}$, we have:
\begin{align*}
&\sum_{j=0}^{n}(-1)^{n-j}\mathcal{P}^{\mu_{n-j}}\bigg(   \sum_{i_1 < i_2 \ldots <i_{j}}^n  [\ldots\mathcal{P}x_{i_1},\ldots,\mathcal{P}x_{i_2},\ldots\mathcal{P}x_{i_{j}},\ldots]\bigg)\\
=&\sum_{j=0}^{n}(-1)^{n-j}\bigg(   \sum_{i_1 < i_2 \ldots <i_{j}}^n  \mathcal{P}^{\mu_{n-j}} \mathcal{P}^{\mu_{j}} [x_{1},\ldots,x_{n}]\bigg)\\
=&\bigg(\sum_{j=0}^{n}(-1)^{n-j}\binom{n}{j} \bigg)  \mathcal{P}^{\mu_{n}}  [x_{1},\ldots,x_{n}]\\
=&0.
\end{align*}
 Thus $\mathcal{P}$ is  a product structure.\\
 Fix $i$ such that  $0<i<n$ and let $(k,l)$ such that  $0<k\leq i <l\leq n$ with  $x_k\in  \mathfrak{g}_{+}$ and $x_l \in  \mathfrak{g}_{-}$. According to Eq. \eqref{centroid}, we get
 \begin{align*}
     &\mathcal{P}[x_{1},\ldots,x_{i},x_{i+1},\ldots,x_{n}] = [\mathcal{P}x_{1},\ldots,x_{i},\ldots,x_{i}] = [x_{1},\ldots,x_{i},\ldots,x_{i}].
 \end{align*}
 On the other hand, by Eq. \eqref{centroid1} we have
  \begin{align*}
     \mathcal{P}[x_{1},\ldots,x_{i},x_{i+1},\ldots,x_{i}]
     =&  [x_{1},\ldots,x_{i},\mathcal{P}x_{i+1},\ldots,x_{i}]\\
     =& - [x_{1},\ldots,x_{i},x_{i+1},\ldots,x_{i}].
 \end{align*}
 Then, we obtain
$$[\mathfrak{g}_{+}, \ldots, \mathfrak{g}_{+}, \underbrace{\mathfrak{g}_{+}}_{i\text{-th place}} \mathfrak{g}_{-}, \ldots, \mathfrak{g}_{-}] =0. $$
\end{proof}
\begin{prop}
Let $\mathcal{G}=(\mathfrak{g},[\cdot,\cdots,\cdot],\varepsilon,\alpha)$ be a $n$-Hom-Lie color algebra. Then $\mathcal{G}$ has a strict product structure, if and
only if, $\mathfrak{g}=\bigoplus_{\gamma\in \Gamma}\mathfrak{g}_{\gamma}$, admits a decomposition 
$
    \mathfrak{g}= \mathfrak{g}_{ +} \oplus \mathfrak{g}_{-},
$
where $\mathfrak{g}_{+}$ and $\mathfrak{g}_{ -} $ are graded subalgebras of $\mathfrak{g}$ such that $[\mathfrak{g}_{ +}, \ldots, \mathfrak{g}_{ +}, \underbrace{\mathfrak{g}_{ +}}_{i\text{-th place}}, \mathfrak{g}_{ -}, \ldots, \mathfrak{g}_{ -}] =0 $, $\forall~ 1\leq i \leq n-1$.
\end{prop}
\begin{proof} The first implication is a direct computation from  Lemma \ref{Lemme-strict}. Conversely, using Theorem \ref{thm-decomposition}, the map $\mathcal{P} $  defined  in Eq.\eqref{eq-product} is an almost structure and, for all $x_{k} = x_{k}^{+} +  x_{k}^{-} \in \mathfrak{g}$, where $x_k^+\in \mathfrak{g}_+ $ and $x_k^-\in \mathfrak{g}_- $. We have
\begin{align*}
    \mathcal{P}[x_1,x_2,\ldots,x_n] &= \mathcal{P}[x_{1}^{+} +  x_{1}^{-},x_{2}^{+} +  x_{2}^{-},\ldots,x_{n}^{+} +  x_{n}^{-}] \\
    &= \mathcal{P}[x_{1}^{+},x_{2}^{+},\ldots,x_{n}^{+} ]  + \mathcal{P}[x_{1}^{-},x_{2}^{-},\ldots,x_{n}^{-} ] \\
     &= [x_{1}^{+},x_{2}^{+},\ldots,x_{n}^{+} ] - [x_{1}^{-},x_{2}^{-},\ldots,x_{n}^{-} ] \\
      & =[\mathcal{P}x_{1}^{+},x_{2}^{+},\ldots,x_{n}^{+} ]  +[ \mathcal{P}x_{1}^{-},x_{2}^{-},\ldots,x_{n}^{-} ] \\
       & =[\mathcal{P}x_{1},x_{2},\ldots,x_{n} ].
\end{align*}
Then $\mathcal{P}$ is a strict product structure on  $\mathcal{G}$.
\end{proof}

\begin{exa}
Let $\mathfrak{g}_\alpha=(\mathfrak{g}, [\cdot, \cdot, \cdot, \cdot]_\alpha, \varepsilon, \alpha)$ be the  $4$-Hom-Lie color algebra defined in  Example \ref{Exple}. Define a linear map of degree zero $\mathcal{P}: \mathfrak{g} \to \mathfrak{g}$ by $$\mathcal{P}(e_1) = e_1,\quad  \mathcal{P}(e_2) = e_2, \quad \mathcal{P}(e_3) = -e_3,\quad  \mathcal{P}(e_4) = -e_4 \quad  \text{and} \quad  \mathcal{P}(e_5) = -e_5. $$
It easy to prove that $\mathcal{P}$ is a strict product structure, therefore it  is a product structure.
Using Theorem \ref{thm-decomposition}, we deduce that the  graded subalgebras $\mathfrak{g}_+$  and $\mathfrak{g}_-$ are generated by $<e_1, e_2>$ and $<e_3, e_4,e_5>$ respectively. Thus
$$\mathfrak{g}_+=\mathfrak{g}_{(0, 0)}\ \ \text{and}\ \ \mathfrak{g}_-= \mathfrak{g}_{(0, 1)}\oplus \mathfrak{g}_{(1, 0)}\oplus \mathfrak{g}_{(1, 1)}.$$

\end{exa}

\end{document}